\documentclass{mcom-l}

\providecommand{\keywords}[1]{\textbf{\textit{Index terms---}} #1}
\usepackage{color}
\usepackage{CJK,indentfirst,amsmath,amsfonts,amssymb,amsthm}
\usepackage{amsfonts}
\usepackage{amssymb}
\usepackage{amsmath}
\usepackage{latexsym}
\usepackage{graphicx}
\usepackage{epstopdf}
\usepackage{geometry}
\usepackage{hyperref}
\usepackage{color}
\usepackage{float}
\usepackage{algorithm}
\usepackage{algorithmic}
\usepackage{multicol}
\usepackage[bottom]{footmisc}
\usepackage{url}
\usepackage{bm}
\usepackage{subfigure}
\usepackage{amsmath}
\usepackage{bm}
\usepackage{CJK}
\usepackage{enumerate}
\usepackage{color}
\usepackage{amsthm}
\usepackage{algorithm}
\usepackage{algorithmic}
\usepackage{geometry}
\usepackage{amsfonts}
\usepackage{indentfirst}
\usepackage{mathtools}
\usepackage{setspace}
\usepackage{microtype}
\usepackage{enumerate}
\usepackage[all,tips]{xy}
\SelectTips{cm}{10}
\usepackage{aliascnt}

\usepackage{amsmath}
\numberwithin{equation}{section}

\newtheorem{theorem}{Theorem}[section]
\newtheorem{lemma}[theorem]{Lemma}
\newtheorem{thm}[theorem]{Theorem}

\newtheorem{remark}[theorem]{Remark}

\numberwithin{equation}{section}
\newcommand{\sfrac}[2]{\mathchoice
	{\kern0em\raise.5ex\hbox{\the\scriptfont0 #1}\kern-.15em/
		\kern-.15em\lower.25ex\hbox{\the\scriptfont0 #2}}
	{\kern0em\raise.5ex\hbox{\the\scriptfont0 #1}\kern-.15em/
		\kern-.15em\lower.25ex\hbox{\the\scriptfont0 #2}}
	{\kern0em\raise.5ex\hbox{\the\scriptscriptfont0 #1}\kern-.2em/
		\kern-.15em\lower.25ex\hbox{\the\scriptscriptfont0 #2}}
	{#1\!/#2}}

\def\ub {\bm{u}}
\def\vb {\bm{v}}

\usepackage{graphicx}
\newtheorem{my assumption}{Assumption}

\usepackage{amsthm}

\begin{document}
	
	\title{An Iterative Decoupled Algorithm with Unconditional Stability for Biot Model}
	
	
	\author{Huipeng Gu}
	\address{Department of Mathematics, Southern University of Science and Technology, Shenzhen, Guangdong 518055, China.}
	\email{12131226@mail.sustech.edu.cn}
	
	\author{Mingchao Cai}
	\address{Department of Mathematics, Morgan State University, Baltimore, MD 21251, USA}
	\email{cmchao2005@gmail.com}
	\thanks{The work of the first and the third author was partially supported by the NSF of China No. 11971221 and the Shenzhen Sci-Tech Fund No. RCJC20200714114556020, JCYJ20200109115422828, JCYJ20190809150413261 and JCYJ20170818153840322, and Guangdong Provincial Key Laboratory of Computational Science and Material Design No. 2019B030301001. The second author gratefully acknowledge supports by NIH BUILD grant through UL1GM118973, NIH-RCMI grant through U54MD013376, and the National Science Foundation awards (1700328, 1831950).}

	\author{Jingzhi Li}
	\address{Department of Mathematics, Southern University of Science and Technology, Shenzhen, Guangdong 518055, China.}
	\email{li.jz@sustech.edu.cn}

	\subjclass[2010]{Primary }
	
	\date{}
	
	\dedicatory{}
	
	\begin{abstract}
	This paper is concerned with numerical algorithms for Biot model. By introducing an intermediate variable, the classical 2-field Biot model is written into a 3-field  formulation. Based on such a 3-field formulation, we propose a coupled algorithm, some time-extrapolation based decoupled algorithms, and an iterative decoupled algorithm. Our focus is the analysis of the iterative decoupled algorithm. It is shown that the convergence of the iterative decoupled algorithm requires no extra assumptions on physical parameters or stabilization parameters. Numerical experiments are provided to demonstrate the accuracy and efficiency of the proposed method.
	\end{abstract}
	
	\maketitle
	


\section{Introduction} \label{Introduction}

Poroelasticity describes the interaction between a pore-structured solid and a fluid where the solid is saturated. Its theoretical basis was initially established by Biot \cite{biot1941general, biot1955theory}. Due to its importance, Biot model has been widely used in various fields \cite{kim2011stability, ju2020parameter}, ranging from petroleum engineering to biomedical engineering. Since Biot model is a multiphysics model and the domain is usually irregular, it is not easy to obtain an analytical solution. Thus, many researchers pay their attentions to numerical solutions \cite{korsawe2005least, korsawe2006finite, cai2015comparisons, lee2016robust, yi2017iteratively, lee2019unconditionally, naumovich2006finite}. In many existing works \cite{both2017robust, naumovich2006finite, rodrigo2016stability}, solid displacement and fluid pressure are taken as the primary variables in the Biot model. Correspondingly, these methods are based on the 2-field formulation. However, it is pointed out that elasticity locking and pressure oscillation are the major difficulties for the 2-field formulation based model \cite{yi2017study, lee2017parameter, cai2015comparisons}. To overcome these difficulties, stabilizations and various 3-field or 4-field reformulations are used \cite{rodrigo2016stability, lee2017parameter, oyarzua2016locking, yi2017iteratively}. Following \cite{oyarzua2016locking, lee2017parameter}, an intermediate variable, called  ``total pressure", is introduced to develop a 3-field formulation for Biot model in this paper. By using such a reformulation, one can view the Biot model as a combination of a generalized Stokes problem and a reaction-diffusion problem for the fluid pressure. The advantages of this reformulation are as follows. Firstly, the reformulation enables one to apply the classical inf-sup stable Stokes finite element pairs and the traditional Lagrange elements for the parabolic type reaction-diffusion equation \cite{oyarzua2016locking, lee2017parameter, ju2020parameter, feng2018multiphysics}. Thus, sophisticated discretization is avoided. Secondly, for either a coupled algorithm or a decoupled algorithm based on such a reformulation, some existing fast solvers like multigrid methods and domain decomposition methods can be directly called. Moreover, it has been shown that such a 3-field reformulation enables one to overcome the above mentioned numerical difficulties \cite{oyarzua2016locking, lee2017parameter, ju2020parameter}.

Actually, no matter a 2-field or a 3-field formulation is adopted, numerical methods for Biot model can be classified into three types as follows. 1. Coupled (or monolithic) algorithms, in which all variables are solved together. 2. Decoupled (or called ``partitioned") time-stepping algorithms, in which the numerical computations of different variables are realized by employing the numerical solutions of previous time-steps, see for example \cite{feng2018multiphysics, ju2020parameter}. 3. Iterative algorithms, in which the numerical computations of different variables are realized by applying the solutions of previous iterations. Some well-known iterative methods \cite{both2017robust, kim2011stability, yi2017iteratively} include the drained split, the undrained split, the fixed-strain split,  and the fixed-stress split. In this work, based on the 3-field reformulation \cite{oyarzua2016locking, lee2017parameter, ju2020parameter}, we consider all these three types of methods: a coupled algorithm, some time-extrapolation type partitioned algorithm, and an iterative decoupled algorithm. The coupled algorithm is to solve the fully coupled system, that is, the generalized Stokes problem and the reaction-diffusion problem are put together. In the time-extrapolation based decoupled algorithms, we separate the original problem into two sub-problems and apply the solutions of the previous time step to decouple the computation. However, the time-extrapolation based decoupling will cause stability constraints and accuracy issues. Thus, we propose an iterative decoupled algorithm. Specifically, we employ the time extrapolation and apply iterations for each submodel to improve the solution accuracy in each time step. Such an idea was inspired by the work \cite{both2017robust}, in which a 2-field formulation is adopted and the fixed-stress split is applied for iterations. In this work, we call our iterative method as a ``decoupled algorithm" in the sense that numerical computations for different submodels are decoupled. In other works, some researchers called their algorithms as ``iterative coupling algorithms" in the sense that the physics of different submodels are coupled together. To ensure the convergence of the iterative method in \cite{both2017robust}, they require that the stabilization parameter should be large enough. Compared with the fixed-stress splitting iterative method proposed in \cite{both2017robust}, our iterative decoupled method does not require any stabilization parameter and is unconditionally convergent to the solution of the coupled algorithm. Furthermore, we do not need extra assumptions on physics parameters, particularly for the storage coefficient $c_0$. We comment here that theoretical analysis for many existing iterative methods is valid only when $c_0>0$. Numerical experiments are provided to validate the effectiveness and efficiency of our algorithms.
	
The rest of this paper is structured as follows. In Section \ref{pde_model}, we briefly introduce the (quasi-static) Biot model and present a 3-field reformulation. In Section \ref{fdp}, a coupled algorithm, some time extrapolation based algorithms, and an iterative decoupled algorithm are proposed based on the 3-field formulation. The error analysis of the coupled algorithm is provided in Appendix A. In Section \ref{ca}, we prove that the solution based on the iterative decoupled algorithm converges to that of the coupled algorithm. Numerical experiments are presented in Section \ref{nr}, and conclusions are drawn in Section \ref{conclusion}.

\section{Mathematical formulations}\label{pde_model}

\subsection{The Biot model and its reformulation}
Let $\Omega \subset \mathbb{R}^d$ ($d= 2$ or $3$) be a bounded polygonal domain with boundary $\partial \Omega$. We use $( \cdot , \cdot )$ and $\langle \cdot , \cdot \rangle$ to denote the standard $L^2(\Omega)$ and $L^2(\partial \Omega)$ inner products, respectively. We will also use the following notations: the standard Sobolev spaces \cite{girault2012finite}, $W^{m,p}(\Omega) = \{ u \ | \ D^\alpha u \in L^p(\Omega) , 0 \leq \alpha \leq m , \|u\|_{W^{m,p}}<\infty \}$; $H^m(\Omega)$ for $W^{m,2}(\Omega)$, and $\| \cdot \|_{H^m(\Omega)}$ for $\| \cdot \|_{W^{m,2}(\Omega)}$; $H^m_{0,\Gamma} (\Omega)$ for the subspace of $H^m(\Omega )$ with the vanishing trace on $\Gamma \subset \partial \Omega$.
	
The classical 2-field formulation of Biot model is given as follows
\begin{align}
	& - \mbox{div} \bm{\sigma} (\bm{u}) +\alpha \nabla p = \bm{f}, \label{twofield1} \\
	& (c_0 p + \alpha \mbox{div}\bm{u})_t - \mbox{div} K (\nabla p - \rho_f \bm{g}) = Q_s. \label{twofield2}
\end{align}
Here, equation \eqref{twofield1} is the momentum equation, and equation  \eqref{twofield2} describes the conservation of mass for fluid flow in porous media. In the above equations, the primary unknowns are the displacement vector of the solid phase $\bm{u}$ and the pressure of the fluid phase $p$. The coefficients $\alpha$ is the Biot-Willis constant which is close to 1, $\bm{f}$ is the body force, $c_0$ is the specific storage coefficient, $K$ represents the hydraulic conductivity, $\rho_f$ is the fluid density, $\bm{g}$ is the gravitational acceleration, $Q_s$ is a source or sink term,
\begin{align*}
	\sigma(\bm{u})=2\mu\varepsilon(\bm{u})+\lambda \mbox{div}\bm{u}  \bm{I}, \ \ \ \varepsilon(\bm{u}) = \frac{1}{2}[\nabla \bm{u} + (\nabla \bm{u})^T],
\end{align*}
$\bm{I}$ is the identity matrix, $\lambda$ and $\mu$ are Lam\'{e} constants, which can be expressed in terms of the Young's modulus $E$ and the Poisson ratio $\nu$:
\begin{align*}
	\lambda = \frac{E\nu}{(1+\nu)(1-2\nu)},\ \ \ \mu = \frac{E}{2(1+\nu)}.
\end{align*}
Proper boundary and initial conditions should be provided in order to ensure the existence and uniqueness of the solution. In this paper, we consider a mixed partial Neumann and partial Dirichlet conditions: assuming $\partial\Omega=\Gamma_d\cup\Gamma_t=\Gamma_p\cup\Gamma_f$ with $|\Gamma_d| >0$ and $|\Gamma_p|>0$ . Here, $\Gamma_d$ and $\Gamma_p$ denote the Dirichlet boundary for $\bm{u}$ and $p$, respectively; $\Gamma_t$ and $\Gamma_f$ denote the Neumann boundary for $\bm{u}$ and $p$, respectively. For instance,
\begin{align*}
	\bm{u} = \bm{0},       \quad & \mbox{on} \ \Gamma_d,  \\
	\bm{\sigma}(\bm{u}) \bm{n}+\alpha p \bm{n} = \bm{h},  \quad & \mbox{on} \ \Gamma_t,  \\
	p=0,    \quad & \mbox{on} \ \Gamma_p,  \\
	K(\nabla p - \rho_f \bm{g}) \cdot \bm{n} = g_2,    \quad & \mbox{on} \ \Gamma_f,
\end{align*}
where $\bm{n}$ is the unit outward normal to the boundary. Without loss of generality, the above Dirichlet boundary conditions are assumed to be homogeneous. For ease of presentation, we assume that $\bm{g} = \bm{0}$, $\bm{f}$, $\bm{h}$, $Q_s$, $g_2$ all are assumed to be independent of $t$. The initial conditions are given
\begin{align*}
	\bm{u}(0) = \bm{u}_0, \ \ \    p(0) = p_0.
\end{align*}
Following \cite{lee2017parameter, ju2020parameter}, we introduce  the so-called ``total pressure": $\xi=\alpha p -\lambda \mbox{div} \bm{u}$. The corresponding initial condition is $\xi_0 = \alpha p_0 -\lambda \mbox{div} \bm{u}_0$. Then, \eqref{twofield1}-\eqref{twofield2} can be written as
\begin{align}
	& -2 \mu \mbox{div} (\bm{\varepsilon}(\bm{u})) + \nabla \xi = \bm{f}, \label{threefield1} \\
	& - \mbox{div} \bm{u} - \frac{1}{\lambda} \xi + \frac{\alpha}{\lambda} p = 0, \label{threefield2} \\
	& ((c_0 + \frac{\alpha^2}{\lambda})p - \frac{\alpha}{\lambda}\xi) _t - \mbox{div} K (\nabla p - \rho_f \bm{g}) = Q_s. \label{threefield3}
\end{align}
After such a reformulation, the above boundary conditions and initial conditions can still be applied to the model \eqref{threefield1}-\eqref{threefield3}.

In order to study the variational problem for the 3-field formulation \eqref{threefield1}-\eqref{threefield3},  we introduce the following functional spaces: $\bm{V} \coloneqq \{ \bm{v} \in  \bm{H}^1(\Omega); \bm{v}|_{\Gamma_d} = 0\}$, $W \coloneqq L^2(\Omega)$, and $M \coloneqq \{ \psi \in H^1 (\Omega); \psi|_{\Gamma_p}=0\}$. Their dual spaces are denoted as $\bm{V}^\prime$, ${W}^\prime$ and $M^\prime$. Given that $|\Gamma_d|>0$, the Korn's inequality \cite{korn1991nitsche} holds on $\bm{V}$, that is, there exists a constant $C_k = C_k(\Omega,\Gamma_d) > 0$ such that
\begin{align}
	\| \bm{u} \|_{H^1(\Omega)} \leq C_k \| \varepsilon(\bm{u}) \|_{L^2(\Omega)}, \ \forall \bm{u} \in \bm{V}. \label{korncon}
\end{align}
Furthermore, the following inf-sup condition \cite{brenner1993nonconforming} holds: there exists a constant $\beta_0 > 0$ depending	only on $\Omega$ and $\Gamma_d$ such that
\begin{align}
	\sup_{\bm{u}\in \bm{V}} \frac{(\mbox{div} \bm{u} , q)}{\| \bm{u} \|_{H^1(\Omega)}} \geq \beta_0 \| q \|_{L^2(\Omega)}, \ \forall q \in L^2(\Omega). \label{infsupcon}
\end{align}

\begin{my assumption}\label{assumption1} We assume that $\bm{u}_0 \in \bm{H}^1(\Omega)$, $\bm{f} \in \bm{L}^2(\Omega)$, $\bm{h} \in \bm{L}^2(\Gamma_t)$, $p_0 \in L^2(\Omega)$, $Q_s \in L^2(\Omega)$ and $g_2 \in L^2(\Gamma_f)$. We also assume that $\mu>0$, $\lambda > 0$, $K$ is uniformly bounded from the above and below, $c_0 \geq 0$, $T > 0$.
\end{my assumption}
	
For simplicity, we will assume \textbf{\bf{Assumption \ref{assumption1}}} holds in the rest of our paper. For ease of presentation, we assume that $\bm{g} = \bm{0}$, $\bm{f}$, $\bm{h}$, $Q_s$, and $g_2$ are independent of $t$. Given $T>0$, a 3-tuple $(\bm{u},\xi,p)\in \bm{V} \times W \times M$ with
\begin{align*}
	& \bm{u} \in L^{\infty}(0,T;\bm{V}), \xi \in L^{\infty}(0,T;W), \\
	& p \in L^{\infty}(0,T;L^2(\Omega)) \cap L^{2}(0,T;M), \\
	& p_t,\xi_t \in L^2(0,T;M^\prime),
\end{align*}
is called a weak solution of problem \eqref{threefield1}-\eqref{threefield3}, if there holds
\begin{align}
	& 2 \mu (\varepsilon(\bm{u}),\varepsilon(\bm{v})) - (\xi,\mbox{div}\bm{v}) = (\bm{f},\bm{v})+ \langle \bm{h},\bm{v} \rangle_{\Gamma_t}, \ \forall \bm{v}\in \bm{V}, \label{vp1}  \\
	& - (\mbox{div} \bm{u} ,\phi) - \frac{1}{\lambda} (\xi,\phi) + \frac{\alpha}{\lambda} (p,\phi) = 0, \ \forall \phi \in W, \label{vp2} \\
	& (((c_0 + \frac{\alpha^2}{\lambda})p - \frac{\alpha}{\lambda}\xi) _t,\psi) + K (\nabla p,\nabla \psi) = (Q_s,\psi)+ \langle g_2,\psi \rangle_{\Gamma_f}, \ \forall \psi \in M, \label{vp3}
\end{align}
for almost every $t\in [0,T]$.
	
\subsection{Energy estimates}
The following lemma describes the energy law for problem \eqref{vp1}-\eqref{vp3}.
		
\begin{lemma} \label{lemma21}
Every weak solution $(\bm{u},\xi,p)$ of problem \eqref{vp1}-\eqref{vp3} satisfies the following energy law:
\begin{align}
	E(t) + \int_0^t  K (\nabla p,\nabla p) ds - \int_0^t (Q_s,p) ds - \int_0^t \langle g_2,p \rangle_{\Gamma_f} ds = E(0), \label{lemma21eq}
\end{align}
for all $t \in (0,T]$, where
\begin{align}
	E(t) \coloneqq \mu \| \varepsilon(\bm{u}(t))\|_{L^2(\Omega)}^2 + \frac{1}{2\lambda}\|\alpha p(t) - \xi(t)\|_{L^2(\Omega)}^2 + \frac{c_0}{2} \|p(t)\|_{L^2(\Omega)}^2 - (\bm{f},\bm{u}(t)) -\langle \bm{h},\bm{u}(t) \rangle_{\Gamma_t}. \nonumber
\end{align}
Moreover,
\begin{align}
	\|\xi(t)\|_{L^2(\Omega)} \leq C ( 2\mu\|\varepsilon(\bm{u}(t))\|_{L^2(\Omega)}+\|\bm{f}\|_{L^2(\Omega)}+\|\bm{h}\|_{L^2(\Gamma_t)}), \nonumber
\end{align}
where $C= C_k/\beta_0$ is a constant depending only on $\Omega$ and $\Gamma_d$.
\end{lemma}
\begin{proof}[Proof.]
We use the standard techniques to show the results. Setting $\bm{v}=\bm{u}_t$ in \eqref{vp1}, $\psi=p$ in \eqref{vp3}, differentiating with respect to $t$ in \eqref{vp2} and setting $\phi=\xi$, we have
\begin{align*}
	& 2 \mu (\varepsilon(\bm{u}),\varepsilon(\bm{u}_t)) - (\xi,\mbox{div}\bm{u}_t) = (\bm{f},\bm{u}_t)+ \langle \bm{h},\bm{u}_t \rangle_{\Gamma_t}, \nonumber \\
	& (\mbox{div} \bm{u}_t ,\xi) + \frac{1}{\lambda} (\xi_t,\xi) - \frac{\alpha}{\lambda} (p_t,\xi) = 0, \nonumber   \\
	& (c_0 + \frac{\alpha^2}{\lambda})(p_t,p) - \frac{\alpha}{\lambda}(\xi_t,p) + K (\nabla p,\nabla p) = (Q_s,p)+ \langle g_2,p \rangle_{\Gamma_f}.\nonumber
\end{align*}
Adding the above three equations together, we obtain that
\begin{align}
	& 2 \mu (\varepsilon(\bm{u}),\varepsilon(\bm{u}_t))  + \frac{1}{\lambda} (\xi_t,\xi) + (c_0 + \frac{\alpha^2}{\lambda})(p_t,p) - \frac{\alpha}{\lambda}(\xi_t,p) - \frac{\alpha}{\lambda} (p_t,\xi) + K (\nabla p,\nabla p) \nonumber \\
	= & \ (\bm{f},\bm{u}_t)+ \langle \bm{h},\bm{u}_t \rangle_{\Gamma_t} +(Q_s,p)+ \langle g_2,p \rangle_{\Gamma_f}. \label{e1}
\end{align}
Since $(\alpha p_t - \xi_t,\alpha p - \xi) = \alpha^2(p_t,p) - \alpha(\xi_t,p) - \alpha(p_t,\xi) + (\xi_t,\xi)$, \eqref{e1} can be rewritten as
\begin{align}
	& 2 \mu (\varepsilon(\bm{u}),\varepsilon(\bm{u}_t))  + \frac{1}{\lambda}(\alpha p_t - \xi_t,\alpha p - \xi) + c_0 (p_t,p) + K (\nabla p,\nabla p) \nonumber \\
	& = (\bm{f},\bm{u}_t)+\langle\bm{h},\bm{u}_t\rangle_{\Gamma_t} +(Q_s,p)+\langle g_2,p \rangle_{\Gamma_f}. \label{e2}
\end{align}
Integrating \eqref{e2} in $t$ over the interval $(0,s)$ for any $s \in (0,T]$, we derive \eqref{lemma21eq}. The bound for $\xi$ follows from the inf-sup condition and the Korn's inequality. Specifically, from \eqref{vp1}, we see that the following inequality holds
\begin{align}
	\beta_0 \| \xi \|_{L^2(\Omega)} & \leq \sup \limits_{\bm{v} \in \bm{V}} \frac{|(\mbox{div} \bm{v},\xi(t) )|}{\|\bm{v}\|_{H^1(\Omega)}} \nonumber \\
	& \leq \sup \limits_{\bm{v} \in \bm{V}}  \frac{|2\mu(\varepsilon(\bm{u} ),\varepsilon(\bm{v}))|+|(\bm{f},\bm{v})|+|\langle\bm{h},\bm{v} \rangle_{\Gamma_t}|}{\|\bm{v}\|_{H^1(\Omega)}} \nonumber \\
	& \leq C_k (2\mu \|\varepsilon(\bm{u} )\|_{L^2(\Omega)}+\|\bm{f}\|_{L^2(\Omega)}+\|\bm{h}\|_{L^2(\Gamma_t)}) \label{bound}.
\end{align}
The constant $\beta_0$ is from the inf-sup condition \eqref{infsupcon} and $C_k$ is from the Korn's inequality \eqref{korncon}. This completes the proof.
\end{proof}
	
The energy law \eqref{lemma21eq} implies the following priori estimate immediately.
\begin{thm}
Let $(\bm{u},\xi,p)$ be the solution of problem \eqref{vp1}-\eqref{vp3}, there holds
\begin{align}
	\sqrt{2\mu} \| \varepsilon(\bm{u})\| & _{L^\infty(0,T;L^2(\Omega))} + \sqrt{\frac{1}{\lambda}}\|\alpha p - \xi\|_{L^\infty(0,T;L^2(\Omega))} \nonumber \\
	+ & \sqrt{c_0} \|p\|_{L^\infty(0,T;L^2(\Omega))} + \sqrt{2K} \|\nabla p\|_{L^2(0,T;L^2(\Omega))} \leq C,
\end{align}
where $C = C( \|\bm{u}_0\|_{H^1(\Omega)}, \|p_0\|_{L^2(\Omega)}, \|\bm{f}\|_{L^2(\Omega)}, \|\bm{h}\|_{L^2(\Gamma_t)}, \|Q_s\|_{L^2(\Omega)}, \|g_2\|_{L^2(\Gamma_f)}) $ is a positive constant.
\end{thm}


\section{Numerical algorithms}\label{fdp}
	
We apply the Taylor-Hood elements for the pair $(\bm{u}, \xi)$, i.e., $(\bm{P}_2,P_1)$ Lagrange finite elements, and $P_1$ Lagrange finite elements for the fluid pressure $p$. Then, the finite element spaces are
	\begin{equation} \label{fespace}
		\begin{split}
			& \bm{V}_{h}:= \{ \bm{v}_h \in {\bf C}^0(\bar{\Omega});~\vb_h|_{\Gamma_d}=0, ~\bm{v}_h |_K  \in {\bm P}_2(K), ~\forall K \in T_h \},\\
			& W_h := \{\phi_h \in C^0(\bar{\Omega}); \phi_h |_K  \in  P_1(K), ~\forall K \in T_h \},\\
			& M_h := \{ \psi_h \in C^0(\bar{\Omega}); ~ \psi|_{\Gamma_p}=0,~\psi_h |_K  \in P_1(K), ~\forall K \in T_h  \}.
		\end{split}
	\end{equation}
We note that $\bm{V}_{h} \times W_h$ is a stable Stokes pair, i.e.,
there exists a constant $\beta^{*}_0 > 0$,  independent of $h$, such that
\begin{align}
	\sup_{u_h \in \bm{V}_h} \frac{(\mbox{div} \bm{u}_h , q)}{\| \bm{u}_h \|_{H^1(\Omega)}} \geq \beta^{*}_0 \| q \|_{L^2(\Omega)}, \ \forall q \in L^2(\Omega).
\end{align}

An equidistant partition $0 = t_0 < t_1 < \cdots < t_N = T$ with a step size $\Delta t$ is considered for the time discretization. For simplicity, we define $\bm{u}^n \coloneqq \bm{u}(t^n)$, $\xi^n \coloneqq \xi(t^n)$, and $p^n \coloneqq p(t^n)$.

\subsection{A coupled algorithm and some time-extrapolation based decoupled algorithms}
	
Suppose that initial values $(\bm{u}^0_h,\xi^0_h,p^0_h) \in \bm{V}_{h} \times W_h \times M_h$ are provided, we apply a backward Euler scheme for the time discretization to \eqref{vp3}. Let us consider the following algorithms: for all $n\in N$, given $(\bm{u}^{n-1}_h,\xi^{n-1}_h,p^{n-1}_h) \in \bm{V}_{h} \times W_h \times M_h$, find $(\bm{u}^{n}_h,\xi^{n}_h,p^{n}_h) \in \bm{V}_{h} \times W_h \times M_h$, such that for all $(\bm{v}_h,\phi_h,\psi_h) \in \bm{V}_{h} \times W_h \times M_h$,
\begin{align}
	2 \mu (\varepsilon(\bm{u}^n_h),\varepsilon(\bm{v}_h)) - (\xi^n_h,\mbox{div}\bm{v}_h)  =   (\bm{f},\bm{v}_h) + &\langle \bm{h},\bm{v}_h \rangle_{\Gamma_t},  \label{c1} \\
		(\mbox{div} \bm{u}^n_h ,\phi_h) + \frac{1}{\lambda} (\xi^n_h,\phi_h) - \frac{\alpha}{\lambda} (p^{n-\theta}_h,\phi_h) =  0,  \label{c2} \\
		(c_0 + \frac{\alpha^2}{\lambda})\left( p^n_h , \psi_h \right)- \frac{\alpha}{\lambda}\left( \xi^n_h , \psi_h \right)  + K\Delta t (\nabla p^n_h,\nabla \psi_h)
		= & \Delta t (Q_s,\psi_h) \nonumber \\
		+\Delta t \langle g_2,\psi_h \rangle_{\Gamma_f}
		+ (c_0 + \frac{\alpha^2}{\lambda}) & \left( p^{n-1}_h ,  \psi_h \right)-  \frac{\alpha}{\lambda}\left( \xi^{n-1}_h , \psi_h \right). \label{c3}
	\end{align}
In \eqref{c2}, $\theta = 0$ or $1$. If $\theta =0$, the above algorithm is a coupled algorithm, which was firstly proposed in \cite{oyarzua2016locking}. If $\theta =1$, then the above algorithm is a time-extrapolation based (or semi-implicit) decoupled algorithm, which was firstly proposed in \cite{ju2020parameter} without theoretical analysis.
If $\theta = 0$, equations \eqref{c1}-\eqref{c3} are coupled, therefore a large system contains all variables must be solved together. Instead of solving the Biot problem in a fully coupled manner, one can choose $\theta = 1$ to separate the original problem into two sub-problems, because a generalized Stokes equation for $\bm{u}$ and $\xi$ is obtained if one moves $\frac{\alpha}{\lambda}p$ to the right-hand side of \eqref{threefield2}, and \eqref{threefield3} is a reaction-diffusion problem for $p$ if the term contains $\xi$ is moved to the right hand side. With these observations, one can actually design two time-extrapolation based decoupled algorithms: one is solving for $\bm{u}$ and $\xi$ together firstly and then solving a reaction-diffusion equation for $p$, the other is solving for $p$ firstly, then solving for $\bm{u}$ and $\xi$.
These decoupling strategies will have stability constraints, which require that the time step size should be chosen small enough. Roughly spoken, $\Delta t$ should be of  order $O(h^2)$ \cite{feng2018multiphysics}. This means that time-extrapolation based decoupled algorithms can not guarantee the stability or accuracy if the time step is too large. From now on, for ease of presentation, we will abbreviate the time-extrapolation based decoupled algorithm as the TE decoupled algorithm and will only consider the TE decoupled algorithm which solves $\bm{u}$ and $\xi$ firstly.

\subsection{An iterative decoupled algorithm}
In order to avoid the stability constraints, we propose an iterative decoupled algorithm. In each time step of the algorithm, we use the previous iterates as the initial guess, then solve a reaction-diffusion equation for $p$ and a generalized Stokes equations for $\bm{u}$ and $\xi$ alternately until a convergence is reached. Let us define a sequence $(\bm{u}^{n,i}_h ,\xi^{n,i}_h ,p^{n,i}_h )$ with $i \geq 0$ being the iteration index. After initialization, i.e., $\bm{u}^{n,0}_h = \bm{u}^{n-1}_h$, $\xi^{n,0}_h = \xi^{n-1}_h$, and $p^{n,0}_h = p^{n-1}_h$, each iteration is divided into the following two steps. 
For a fixed $n$, the $i$-th iteration reads as:
	
	\noindent \textbf{Step 1} Given $\xi_h^{n,i-1}\in W_h$, find $p_h^{n,i}\in M_h$ such that
	\begin{align}
		&(c_0 + \frac{\alpha^2}{\lambda})(p^{n,i}_h , \psi_h) + K \Delta t (\nabla p^{n,i}_h,\nabla \psi_h) \nonumber \\
		= & (c_0 + \frac{\alpha^2}{\lambda})(p^{n-1}_h , \psi_h)
		+  \frac{\alpha}{\lambda}(\xi^{n,i-1}_h-\xi^{n-1}_h,\psi_h) + \Delta t (Q_s,\psi_h) +\Delta t \langle g_2,\psi_h \rangle_{\Gamma_f}. \label{d1}
	\end{align}
	\noindent \textbf{Step 2} Given $p_h^{n,i}\in M_h$, find $(\bm{u}_h^{n,i},\xi_h^{n,i})\in \bm{V}_h \times W_h$ such that
	\begin{align}
		& 2 \mu (\varepsilon(\bm{u}^{n,i}_h),\varepsilon(\bm{v}_h)) - (\xi^{n,i}_h,\mbox{div}\bm{v}_h) = (\bm{f},\bm{v}_h)+\langle \bm{h},\bm{v}_h \rangle_{\Gamma_t},  \label{d2} \\
		&  (\mbox{div} \bm{u}^{n,i}_h ,\phi_h) + \frac{1}{\lambda} (\xi^{n,i}_h,\phi_h) = \frac{\alpha}{\lambda} (p^{n,i}_h,\phi_h). \label{d3}
	\end{align}
For simplicity, the backward Euler scheme is chosen for the time discretization of the reaction-diffusion equation \eqref{d1}. Other higher order time-stepping schemes can also be applied here.


\section{Convergence analysis of the iterative decoupled algorithm}\label{ca}
For the error analysis of the coupled algorithm, we refer the readers to the \textbf{\bf{Appendix A}} of this paper. It is shown that the coupled algorithm is unconditionally stable and convergent, and the time error is of order $O(\Delta t)$, the energy-norm errors for $\bm{u}$ and $\xi$ are of order $O(h^2)$, and the energy-norm error for $p$ is of order $O(h)$. In this section, we will show that the sequences $(\bm{u}^{n,i}_h ,\xi^{n,i}_h ,p^{n,i}_h )$ will converge to the solution $(\bm{u}^{n}_h ,\xi^{n}_h ,p^{n}_h )$ of the coupled algorithm if $ i \rightarrow \infty$. We firstly introduce the following lemma \cite{storvik2018optimization}.
	
\begin{lemma} \label{lemma1}
	For all $\bm{u}_h \in \bm{V}_h$, the following inequality holds
	\begin{align}
		\| \mbox{div} \bm{u}_h \|_{L^2(\Omega)} \leq \sqrt{d} \| \varepsilon(\bm{u}_h) \|_{L^2(\Omega)}. \label{lem1}
	\end{align}
\end{lemma}
	
Now, we are in a position to show the main theorem.
	
\begin{thm}
	Let $(\bm{u}^n_h,\xi^n_h,p^n_h)$ and $(\bm{u}^{n,i}_h,\xi^{n,i}_h,p^{n,i}_h)$ be the solutions of problem \eqref{c1}-\eqref{c3} with $\theta = 0$ and problem \eqref{d1}-\eqref{d3}, respectively. Let $e^i_{\bm{u}} = \bm{u}^{n,i}_h - \bm{u}^n_h$, $e^i_{\xi} = \xi^{n,i}_h - \xi^n_h$, and $e^i_{p} = p^{n,i}_h - p^n_h$ denote the errors between the iterative solution in the $i$-th step and the solution of the coupled algorithm. Then, for all $i \geq 1$, it holds that
	\begin{align}
		\|e^i_{\xi}\|_{L^2(\Omega)} \leq C \|e^{i-1}_{\xi}\|_{L^2(\Omega)} ,\label{t1}
	\end{align}
	where $C = \left( \frac{\frac{\alpha^2}{\lambda}}{c_0+\frac{\alpha^2}{\lambda}} \right)^2$ is a positive constant less than or equal to $1$. Moreover,
	\begin{align}
		& \|e^i_p\|_{L^2(\Omega)} \leq  \frac{C}{\alpha} \|e_{\xi}^{i-1}\|_{L^2(\Omega)},\label{t2} \\
		& \|\varepsilon(e_{\bm{u}}^i)\|_{L^2(\Omega)} \leq \frac{\sqrt{d}}{2\mu}\|e^{i}_{\xi}\|_{L^2(\Omega)}. \label{t3}
	\end{align}
    \label{thmimp}
\end{thm}
	\begin{proof}
		Setting $\theta = 0$ in \eqref{c2}, subtracting \eqref{d1}, \eqref{d2} and \eqref{d3} from \eqref{c3}, \eqref{c1} and \eqref{c2}, respectively, we see that
		\begin{align}
			& (c_0 + \frac{\alpha^2}{\lambda})(e^i_p , \psi_h) + K \Delta t (\nabla e^i_p,\nabla \psi_h)
			=  \frac{\alpha}{\lambda}(e_{\xi}^{i-1},\psi_h) , \label{p1} \\
			& 2 \mu (\varepsilon(e_{\bm{u}}^{i}),\varepsilon(\bm{v}_h)) - (e_{\xi}^{i},\mbox{div}\bm{v}_h) = 0,\label{p2} \\
			& (\mbox{div} e_{\bm{u}}^{i} ,\phi_h) + \frac{1}{\lambda} (e_{\xi}^{i},\phi_h) = \frac{\alpha}{\lambda} (e_p^{i},\phi_h), \label{p3}
		\end{align}
		hold for all $(\bm{v}_h,\phi_h,\psi_h) \in \bm{V}_{h} \times W_h \times M_h$. Taking $\psi_h = e_p^i$ in \eqref{p1}, we obtain
		\begin{align}
			(c_0 + \frac{\alpha^2}{\lambda})(e^i_p ,  e_p^i) + K \Delta t \| \nabla e^i_p \|_{L^2(\Omega)}^2 =  \frac{\alpha}{\lambda}(e_{\xi}^{i-1}, e_p^i). \nonumber
		\end{align}
		Discarding the second positive term and apply the Cauchy-Schwarz inequality, we get
		\begin{align}
			(c_0 + \frac{\alpha^2}{\lambda})(e^i_p ,  e_p^i) \leq  \frac{\alpha}{\lambda}(e_{\xi}^{i-1}, e_p^i) \leq \frac{\alpha}{\lambda}\|e_{\xi}^{i-1}\|_{L^2(\Omega)} \|e_p^i\|_{L^2(\Omega)}  \label{con3}.
		\end{align}
The above inequality indicates that $(c_0 + \frac{\alpha^2}{\lambda})\|e^i_p\|_{L^2(\Omega)} \leq  \frac{\alpha}{\lambda}\|e_{\xi}^{i-1}\|_{L^2(\Omega)}$,
which verifies \eqref{t2}.

Taking the test functions in \eqref{p2} and \eqref{p3} as $\bm{v}_h = e_{\bm{u}}^i$ and $\phi_h = e_{\xi}^{i}$, respectively, we obtain the following equations
		\begin{align}
			& 2 \mu (\varepsilon(e_{\bm{u}}^{i}),\varepsilon(e_{\bm{u}}^i)) - (e_{\xi}^{i},\mbox{div}e_{\bm{u}}^i) = 0, \label{con1} \\
			& (\mbox{div} e_{\bm{u}}^{i} ,e_{\xi}^{i}) + \frac{1}{\lambda} (e_{\xi}^{i},e_{\xi}^{i}) = \frac{\alpha}{\lambda} (e_p^{i},e_{\xi}^{i})
			\label{con2}.
		\end{align}
		Summing up \eqref{con1} and \eqref{con2}, and then applying a Cauchy-Schwarz inequality, we have
		\begin{align}
			2\mu \| \varepsilon(e_{\bm{u}}^{i}) \|_{L^2(\Omega)}^2 + \frac{1}{\lambda}\| e_{\xi}^{i}\|_{L^2(\Omega)}^2 = \frac{\alpha}{\lambda} (e_p^{i},e_{\xi}^{i}) \leq \frac{\alpha}{\lambda} \|e_p^{i}\|_{L^2(\Omega)} \|e_{\xi}^{i}\|_{L^2(\Omega)}. \label{4d11}
		\end{align}
		Dropping the first positive term, and using the conclusion of \eqref{con3}, there holds
		\begin{align}
			\|e_{\xi}^{i}\|_{L^2(\Omega)} \leq \alpha \|e_p^{i}\|_{L^2(\Omega)} \leq \frac{\frac{\alpha^2}{\lambda}}{c_0 + \frac{\alpha^2}{\lambda}} \|e_{\xi}^{i-1}\|_{L^2(\Omega)}. \label{4d12}
		\end{align}
		Therefore, \eqref{t1} is proved. Applying \textbf{\bf{Lemma \ref{lemma1}}} to \eqref{con1}, we have
		\begin{align}
			2 \mu \| \varepsilon(e_{\bm{u}}^{i}) \|_{L^2(\Omega)}^2 =  (e_{\xi}^{i},\mbox{div}e_{\bm{u}}^i)
			\leq  \|e_{\xi}^{i}\|_{L^2(\Omega)}\|\mbox{div}e_{\bm{u}}^i\|_{L^2(\Omega)}
			\leq  \sqrt{d} \|e_{\xi}^{i}\|_{L^2(\Omega)}\|\varepsilon(e_{\bm{u}}^i)\|_{L^2(\Omega)}.
		\end{align}
		This yields \eqref{t3}.  The proof is complete.
		
	\end{proof}
	
\begin{remark}
If $c_0>0$, according to \eqref{t1} and the expression of $C$, then $\|e^i_{\xi}\|_{L^2(\Omega)}$ converges to $0$ as $i$ goes to infinity. Following from \eqref{t2} and \eqref{t3}, we see that that $\|e^i_p\|_{L^2(\Omega)}$ and $\|\varepsilon(e_{\bm{u}}^i)\|_{L^2(\Omega)}$ also converge to $0$ if $i$ goes to infinity.
\end{remark}
	
\begin{remark}
   If $c_0 = 0$, we can also prove that the iterative decoupled algorithm is convergent. Note that the arguments in \textbf{\bf{Theorem \ref{thmimp}}}  are valid no matter $c_0$ is $0$ or greater than $0$. Let us assume $c_0=0$ in the following derivation. From \eqref{t1}, we see that $\{\|e^i_{\xi}\|_{L^2(\Omega)}\}$ is still a monotonically non-increasing sequence and has a lower bound. Therefore, $\{\|e^i_{\xi}\|_{L^2(\Omega)}\}$ is convergent. We are going to use the method of contradiction to show that the limit of $\{\|e^i_{\xi}\|_{L^2(\Omega)}\}$ is $0$. If not, let us assume
   $$
   \lim_{i \rightarrow \infty} \|e^i_{\xi}\|_{L^2(\Omega)} = s >0.
   $$
From \eqref{4d11} and \eqref{4d12}, we see that
$$
2\mu \| \varepsilon(e_{\bm{u}}^{i}) \|_{L^2(\Omega)}^2 + \frac{1}{\lambda}\| e_{\xi}^{i}\|_{L^2(\Omega)}^2 \leq \frac{1}{\lambda} \|e_{\xi}^{i-1}\|_{L^2(\Omega)} \|e_{\xi}^{i}\|_{L^2(\Omega)}.
$$
Letting $i \rightarrow \infty$, because $s>0$, it follows that $\lim\limits_{i \rightarrow \infty} \|\varepsilon(e_{\bm{u}}^{i}) \|_{L^2(\Omega)} =0$. Applying the discrete inf-sup condition, and noting from \eqref{p2}, we see that
\begin{align}
\beta_0^* \|e^i_{\xi}\|_{L^2(\Omega)} \le \sup_{\bm{v}_h\in \bm{V}_h} \frac{|(e_{\xi},\mbox{div} \bm{v}_h)|}{\|\bm{v}_h\|_{L^2(\Omega)}} = \sup_{\bm{v}_h\in \bm{V}_h} \frac{|(\varepsilon(e_{\bm{u}}^{i}),\varepsilon(\bm{v}_h))|}{\|\bm{v}_h\|_{L^2(\Omega)}} \lesssim \| \varepsilon(e_{\bm{u}}^{i}) \|_{L^2(\Omega)}. \nonumber
\end{align}
Thus, we derive that $s \leq 0$, which is a contradiction. Therefore, $\lim\limits_{i \rightarrow \infty} \|e^i_{\xi}\|_{L^2(\Omega)} =0$. It follows that
$\lim\limits_{i \rightarrow \infty} \| \varepsilon(e_{\bm{u}}^{i}) \|_{L^2(\Omega)} =0$ and
$\lim\limits_{i \rightarrow \infty} \|e^i_{p}\|_{L^2(\Omega)} =0$.
\label{remark2}
\end{remark}
	
		
\section{Numerical experiments}\label{nr}
	
In this section, we present numerical experiments to compare the accuracy and efficiency of the algorithms described in Section 3. Particularly, we are interested in demonstrating the performance of the different algorithms under various settings of physical parameters. Our tests are based on a 2D benchmark problem with a known analytical solution \cite{ju2020parameter,yi2017study}. All algorithms are implemented in the open-source software package FreeFEM++  \cite{2013New}.

Let $\Omega = [0,1]\times[0,1]$ with $\Gamma_1 = \{ (1,y);0\leq y \leq 1 \} $, $\Gamma_2 = \{ (x,0);0\leq x \leq 1 \} $, $\Gamma_3 = \{ (0,y);0\leq y \leq 1 \} $, and $\Gamma_4 = \{ (x,1);0\leq x \leq 1 \} $. The terminal time is $T = 0.01$. We consider problem \eqref{threefield1}-\eqref{threefield3} with the following force terms and source term:
	\begin{align*}
		& \bm{f} = e^{-t}
		\left(
		\begin{array}{c}
			4 \mu \pi^2 \sin{(2\pi y)}(2\cos{(2\pi x)}-1) + \left( \frac{2\mu \pi^2}{\mu + \lambda}\sin{(\pi x)} + \alpha \pi \cos{(\pi x)}  \right) \sin{(\pi y)} - \pi^2\cos{(\pi(x+y))} \\
			4 \mu \pi^2 \sin{(2\pi x)}(1-2\cos{(2\pi y)}) + \left( \frac{2\mu \pi^2}{\mu + \lambda}\sin{(\pi y)} + \alpha \pi \cos{(\pi y)}  \right) \sin{(\pi x)} - \pi^2\cos{(\pi(x+y))} \\
		\end{array}
		\right) , \\
		& Q_s = e^{-t} \left(  (-c_0+2 \pi^2 K)\sin{(\pi x)}\sin{(\pi y)} - \frac{\alpha \pi}{\mu+\lambda} \sin{(\pi(x+y))} \right).
	\end{align*}
The corresponding boundary conditions and initial conditions are given as:
	\begin{align*}
		p & = e^{-t} \sin{(\pi x)} \sin{(\pi y)},  \quad & \mbox{on}& \ \Gamma_j \times (0,T), j=1,3, \\
		u_1 & = e^{-t}\left( \sin{(2\pi y)}(\cos{(2\pi x)}-1) +\frac{1}{\mu+\lambda}\sin{(\pi x)}\sin{(\pi y)} \right), \quad & \mbox{on}& \ \Gamma_j \times (0,T), j=1,3,\\
		u_2 & = e^{-t}\left( \sin{(2\pi x)}( 1-\cos{(2\pi y)})+ \frac{1}{\mu+\lambda} \sin{(\pi x)}\sin{(\pi y)} \right), \quad & \mbox{on}& \ \Gamma_j \times (0,T), j=1,3,\\
		\bm{h} & = \sigma \bm{n} - \alpha p \bm{n} ,  \quad & \mbox{on}& \ \Gamma_j \times (0,T), j=2,4, \\
		\nabla p \cdot \bm{n} & = e^{-t} (\pi \cos{(\pi x)}\sin{(\pi y)}n_1 + \pi \sin{(\pi x)}\cos{(\pi y)}n_2 ) ,  \quad & \mbox{on}& \ \Gamma_j \times (0,T), j=2,4, \\
		\bm{u}(x,y,0) & = \left(                 
		\begin{array}{c}   
			\sin{(2\pi y)}(\cos{(2\pi x)}-1) +\frac{1}{\mu+\lambda}\sin{(\pi x)}\sin{(\pi y)}  \\  
			\sin{(2\pi x)}( 1-\cos{(2\pi y)})+ \frac{1}{\mu+\lambda} \sin{(\pi x)}\sin{(\pi y)} \\  
		\end{array}
		\right),  \quad & \mbox{in}& \ \Omega, \\
		p(x,y,0) & = \sin{(\pi x)} \sin{(\pi y)},  \quad & \mbox{in} & \ \Omega.
	\end{align*}
	Using the above data, the exact solutions are given as follows:
	\begin{align*}
		& \bm{u}(x,y,t) = e^{-t}
		\left(                 
		\begin{array}{c}   
			\sin{(2\pi y)}(\cos{(2\pi x)}-1) +\frac{1}{\mu+\lambda}\sin{(\pi x)}\sin{(\pi y)}  \\  
			\sin{(2\pi x)}( 1-\cos{(2\pi y)})+ \frac{1}{\mu+\lambda} \sin{(\pi x)}\sin{(\pi y)} \\  
		\end{array}
		\right) , \\
		& p(x,y,t) = e^{-t} \sin{(\pi x)} \sin{(\pi y)}.
	\end{align*}

In the experiments, we use uniform grids with the initial mesh size being $h = 1/16$. The mesh refinement is realized by linking the midpoints of each triangle. The computed $L^2$-norm and $H^1$-norm errors and the convergence rates are reported at the terminal time $T$. We use $iter$ to denote the number of iterations used in the iterative decoupled algorithm. For the TE decoupled algorithm, there is a stability constraint, which state that the time step size should be small enough. Furthermore, as the numerical errors consist of both time error and spatial error, even for the coupled algorithm, time step size should be small enough so that the spatial error is dominant. In our tests, we choose relatively large time step sizes so that we can demonstrate the effectiveness and the efficiency of the iterative decoupled algorithm.

\subsection{Tests for the parameter $\nu$} \label{stest}
In this subsection, we test the performance of the algorithms in Section \ref{fdp} under different settings of the Poisson ratio. The hydraulic conductivity $K$ and the specific storage coefficient $c_0$ are fixed to be $1$.
	
Table \ref{table_coupled} and \ref{table_decoupled} display the results of the coupled algorithm and the TE decoupled algorithm separately. When the mesh size is fine, it is clear that the convergence orders of the TE decoupled algorithm decrease, which is caused by the fact that the time step size is too large and the stability constraint is not satisfied. In comparison, the numerical results exhibited in Table \ref{table_id5} and Table \ref{table_id10} for the iterative decoupled algorithm show that they converge very well. Here, the time step sizes are chosen so that the total operation cost is almost the same as that of the TE decoupled algorithm. More clearly, in our tests, we set $\Delta t$ such that $T / \Delta t \times iter = 10$. By comparing the results in Table \ref{table_id5} with those in Table \ref{table_decoupled}, it is obvious that the iterative decoupled algorithm performs better than the TE decoupled algorithm. The results of Table \ref{table_id10} illustrate that increasing the number of iterations will improve the accuracy of the iterative decoupled algorithm. In addition, we would comment here that if the time step size is small enough, say $\Delta t=1.0 \times10^{-5}$, all algorithms will give energy-norm errors of the optimal orders, although the corresponding numerical results are not reported here.

	\begin{table}[H]
		\begin{center}
			\caption{Convergence rate of the coupled algorithm. $\nu=0.3$ and $\Delta t = 10^{-3}$.}\label{table_coupled}
			\centering
			{\scriptsize
				\begin{tabular}{cclclcl}
					\hline
					1/h                     & \multicolumn{1}{l}{$L^2$\& $H^1$ errors of $\ub$} & \multicolumn{1}{l}{Orders} & \multicolumn{1}{l}{$L^2$ \& $H^1$ errors of $\xi$} & \multicolumn{1}{l}{Orders} & \multicolumn{1}{l}{$L^2$\& $H^1$ errors of $p$} & \multicolumn{1}{l}{Orders} \\ \hline
					16 & 1.063e-03 \& 7.114e-02 &  & 5.297e-03 \& 4.733e-01 &  & 6.091e-03 \& 1.698e-01 &  \\
                    32 & 2.320e-04 \& 1.800e-02 & 2.20 \& 1.98 & 1.267e-03 \& 2.347e-01 & 2.06 \& 1.01 & 1.530e-03 \& 8.476e-02 & 1.99 \& 1.00 \\
                    64 & 5.503e-05 \& 4.528e-03 & 2.08 \& 1.99 & 3.097e-04 \& 1.168e-01 & 2.03 \& 1.01 & 3.792e-04 \& 4.243e-02 & 2.01 \& 1.00 \\
                    128 & 1.296e-05 \& 1.135e-03 & 2.09 \& 2.00 & 7.515e-05 \& 5.825e-02 & 2.04 \& 1.00 & 9.078e-05 \& 2.123e-02 & 2.06 \& 1.00 \\ \hline
				\end{tabular}
			}
		\end{center}
	\end{table}
	
	\begin{table}[H]
		\begin{center}
			\caption{Convergence rate of the TE decoupled algorithm. $\nu=0.3$ and $\Delta t = 10^{-3}$.} \label{table_decoupled}
			\centering
			{\scriptsize
				\begin{tabular}{cclclcl}
					\hline
					1/h                     & \multicolumn{1}{l}{$L^2$\& $H^1$ errors of $\ub$} & \multicolumn{1}{l}{Orders} & \multicolumn{1}{l}{$L^2$ \& $H^1$ errors of $\xi$} & \multicolumn{1}{l}{Orders} & \multicolumn{1}{l}{$L^2$\& $H^1$ errors of $p$} & \multicolumn{1}{l}{Orders} \\ \hline
					16 & 1.046e-03 \& 7.114e-02 &  & 5.263e-03 \& 4.734e-01 &  & 6.087e-03 \& 1.700e-01 &  \\
                    32 & 2.274e-04 \& 1.800e-02 & 2.20 \& 1.98 & 1.253e-03 \& 2.347e-01 & 2.07 \& 1.01 & 1.482e-03 \& 8.479e-02 & 2.04 \& 1.00 \\
                    64 & 1.369e-04 \& 4.568e-03 & 0.73 \& 1.98 & 5.221e-04 \& 1.168e-01 & 1.26 \& 1.01 & 4.039e-04 \& 4.245e-02 & 1.88 \& 1.00 \\
                    128 & 1.448e-04 \& 1.327e-03 & -0.08 \& 1.78 & 4.936e-04 \& 5.831e-02 & 0.08 \& 1.00 & 2.830e-04 \& 2.127e-02 & 0.51 \& 1.00 \\ \hline
				\end{tabular}
			}
		\end{center}
	\end{table}
	
	\begin{table}[H]
		\begin{center}
			\caption{Convergence rate of the iterative decoupled algorithm. $\nu=0.3$, $\Delta t = 5 \times 10^{-3}$, and $iter = 5$.}\label{table_id5}
			\centering
			{\scriptsize
				\begin{tabular}{cclclcl}
					\hline
					1/h                     & \multicolumn{1}{l}{$L^2$\& $H^1$ errors of $\ub$} & \multicolumn{1}{l}{Orders} & \multicolumn{1}{l}{$L^2$ \& $H^1$ errors of $\xi$} & \multicolumn{1}{l}{Orders} & \multicolumn{1}{l}{$L^2$\& $H^1$ errors of $p$} & \multicolumn{1}{l}{Orders} \\ \hline
					16 & 1.070e-03 \& 7.114e-02 &  & 5.308e-03 \& 4.735e-01 &  & 6.120e-03 \& 1.701e-01 &  \\
                    32 & 2.355e-04 \& 1.800e-02 & 2.18 \& 1.98 & 1.281e-03 \& 2.347e-01 & 2.05 \& 1.01 & 1.568e-03 \& 8.482e-02 & 1.96 \& 1.00 \\
                    64 & 5.919e-05 \& 4.531e-03 & 1.99 \& 1.99 & 3.288e-04 \& 1.168e-01 & 1.96 \& 1.01 & 4.284e-04 \& 4.245e-02 & 1.87 \& 1.00 \\
                    128 & 2.198e-05 \& 1.141e-03 & 1.43 \& 1.99 & 1.093e-04 \& 5.825e-02 & 1.59 \& 1.00 & 1.624e-04 \& 2.125e-02 & 1.40 \& 1.00 \\ \hline
				\end{tabular}
			}
		\end{center}
	\end{table}

	\begin{table}[H]
		\begin{center}
			\caption{Convergence rate of the iterative decoupled algorithm. $\nu=0.3$, $\Delta t = 10^{-2}$, and $iter = 10$.}\label{table_id10}
			\centering
			{\scriptsize
				\begin{tabular}{cclclcl}
					\hline
					1/h                     & \multicolumn{1}{l}{$L^2$\& $H^1$ errors of $\ub$} & \multicolumn{1}{l}{Orders} & \multicolumn{1}{l}{$L^2$ \& $H^1$ errors of $\xi$} & \multicolumn{1}{l}{Orders} & \multicolumn{1}{l}{$L^2$\& $H^1$ errors of $p$} & \multicolumn{1}{l}{Orders} \\ \hline
					16 & 1.079e-03 \& 7.115e-02 &  & 5.325e-03 \& 4.739e-01 &  & 6.147e-03 \& 1.704e-01 &  \\
                    32 & 2.316e-04 \& 1.800e-02 & 2.22 \& 1.98 & 1.261e-03 \& 2.348e-01 & 2.08 \& 1.01 & 1.509e-03 \& 8.482e-02 & 2.03 \& 1.01 \\
                    64 & 4.962e-05 \& 4.526e-03 & 2.22 \& 1.99 & 2.950e-04 \& 1.168e-01 & 2.10 \& 1.01 & 3.401e-04 \& 4.243e-02 & 2.15 \& 1.00 \\
                    128 & 6.294e-06 \& 1.133e-03 & 2.98 \& 2.00 & 6.140e-05 \& 5.825e-02 & 2.26 \& 1.00 & 4.909e-05 \& 2.123e-02 & 2.79 \& 1.00 \\ \hline
				\end{tabular}
			}
		\end{center}
	\end{table}
	
The above 4 tables are for the case that the poroelastic material is compressible. In Table 5 to 8, we set the Poisson ratio $\nu = 0.499$ and other physical parameters are not changed. Since the Poisson ratio $\nu$ is close to $0.5$, the poroelastic material is almost incompressible, and the mixed linear elasticity model is close to the incompressible Stokes model. Table \ref{table_coupled2} and \ref{table_decoupled2} are based on the coupled algorithm and the TE decoupled algorithm, respectively.
Table \ref{table_id52} and Table \ref{table_id102} are based on the iterative decoupled algorithm with different numbers of iterations.
Because the Poisson ratio is close to $0.5$, the numerical errors and the corresponding error orders for all algorithms are better than those for $\nu=0.3$. From Table 5 to 8, it is clear that the energy-norm errors based on all algorithm are of the optimal orders. The $L^2$- norm errors based on the TE decoupled algorithm are not of the optimal orders because the time step size is large. By comparing the results in Table \ref{table_id52} and Table \ref{table_id102} with those in Table \ref{table_coupled2} and Table \ref{table_decoupled2}, we again observe that
the iterative decoupled algorithm performs well when the poroelastic material becomes almost incompressible.

	\begin{table}[H]
		\begin{center}
			\caption{Convergence rate of the coupled algorithm. $\nu=0.499$, $\Delta t = 10^{-3}$.}\label{table_coupled2}
			\centering
			{\scriptsize
				\begin{tabular}{cclclcl}
					\hline
					1/h                     & \multicolumn{1}{l}{$L^2$\& $H^1$ errors of $\ub$} & \multicolumn{1}{l}{Orders} & \multicolumn{1}{l}{$L^2$ \& $H^1$ errors of $\xi$} & \multicolumn{1}{l}{Orders} & \multicolumn{1}{l}{$L^2$\& $H^1$ errors of $p$} & \multicolumn{1}{l}{Orders} \\ \hline
					16 & 6.043e-04 \& 7.075e-02 &  & 6.908e-03 \& 7.727e-01 &  & 3.182e-03 \& 1.670e-01 &  \\
                    32 & 7.528e-05 \& 1.786e-02 & 3.00 \& 1.99 & 1.529e-03 \& 3.768e-01 & 2.18 \& 1.04 & 8.048e-04 \& 8.441e-02 & 1.98 \& 0.98 \\
                    64 & 9.360e-06 \& 4.490e-03 & 3.01 \& 1.99 & 3.659e-04 \& 1.870e-01 & 2.06 \& 1.01 & 2.008e-04 \& 4.239e-02 & 2.00 \& 0.99 \\
                    128 & 1.169e-06 \& 1.126e-03 & 3.00 \& 2.00 & 8.965e-05 \& 9.320e-02 & 2.03 \& 1.00 & 4.885e-05 \& 2.123e-02 & 2.04 \& 1.00 \\
                    \hline
				\end{tabular}
			}
		\end{center}
	\end{table}

	\begin{table}[H]
		\begin{center}
			\caption{Convergence rate of the TE decoupled algorithm. $\nu=0.499$, $\Delta t = 10^{-3}$.} \label{table_decoupled2}
			\centering
			{\scriptsize
				\begin{tabular}{cclclcl}
					\hline
					1/h                     & \multicolumn{1}{l}{$L^2$\& $H^1$ errors of $\ub$} & \multicolumn{1}{l}{Orders} & \multicolumn{1}{l}{$L^2$ \& $H^1$ errors of $\xi$} & \multicolumn{1}{l}{Orders} & \multicolumn{1}{l}{$L^2$\& $H^1$ errors of $p$} & \multicolumn{1}{l}{Orders} \\ \hline
					16 & 6.042e-04 \& 7.075e-02 &  & 6.908e-03 \& 7.727e-01 &  & 3.182e-03 \& 1.670e-01 &  \\
                    32 & 7.527e-05 \& 1.786e-02 & 3.00 \& 1.99 & 1.529e-03 \& 3.768e-01 & 2.18 \& 1.04 & 8.048e-04 \& 8.441e-02 & 1.98 \& 0.98 \\
                    64 & 9.374e-06 \& 4.490e-03 & 3.01 \& 1.99 & 3.659e-04 \& 1.870e-01 & 2.06 \& 1.01 & 2.008e-04 \& 4.239e-02 & 2.00 \& 0.99 \\
                    128 & 1.443e-06 \& 1.126e-03 & 2.70 \& 2.00 & 8.969e-05 \& 9.320e-02 & 2.03 \& 1.00 & 4.884e-05 \& 2.123e-02 & 2.04 \& 1.00 \\
                    \hline
				\end{tabular}
			}
		\end{center}
	\end{table}
	
	\begin{table}[H]
		\begin{center}
			\caption{Convergence rate of the iterative decoupled algorithm. $\nu=0.499$, $\Delta t = 5 \times 10^{-3}$, and $iter = 5$.}\label{table_id52}
			\centering
			{\scriptsize
				\begin{tabular}{cclclcl}
					\hline
					1/h                     & \multicolumn{1}{l}{$L^2$\& $H^1$ errors of $\ub$} & \multicolumn{1}{l}{Orders} & \multicolumn{1}{l}{$L^2$ \& $H^1$ errors of $\xi$} & \multicolumn{1}{l}{Orders} & \multicolumn{1}{l}{$L^2$\& $H^1$ errors of $p$} & \multicolumn{1}{l}{Orders} \\ \hline
					16 & 6.043e-04 \& 7.075e-02 &  & 6.908e-03 \& 7.727e-01 &  & 3.170e-03 \& 1.670e-01 &  \\
                    32 & 7.528e-05 \& 1.786e-02 & 3.00 \& 1.99 & 1.529e-03 \& 3.768e-01 & 2.18 \& 1.04 & 7.960e-04 \& 8.441e-02 & 1.99 \& 0.98 \\
                    64 & 9.360e-06 \& 4.490e-03 & 3.01 \& 1.99 & 3.659e-04 \& 1.870e-01 & 2.06 \& 1.01 & 1.930e-04 \& 4.239e-02 & 2.04 \& 0.99 \\
                    128 & 1.168e-06 \& 1.126e-03 & 3.00 \& 2.00 & 8.965e-05 \& 9.320e-02 & 2.03 \& 1.00 & 4.167e-05 \& 2.123e-02 & 2.21 \& 1.00 \\
                    \hline
				\end{tabular}
			}
		\end{center}
	\end{table}

	\begin{table}[H]
		\begin{center}
			\caption{Convergence rate of the iterative decoupled algorithm. $\nu=0.499$, $\Delta t = 10^{-2}$, and $iter = 10$.}\label{table_id102}
			\centering
			{\scriptsize
				\begin{tabular}{cclclcl}
					\hline
					1/h                     & \multicolumn{1}{l}{$L^2$\& $H^1$ errors of $\ub$} & \multicolumn{1}{l}{Orders} & \multicolumn{1}{l}{$L^2$ \& $H^1$ errors of $\xi$} & \multicolumn{1}{l}{Orders} & \multicolumn{1}{l}{$L^2$\& $H^1$ errors of $p$} & \multicolumn{1}{l}{Orders} \\ \hline
					16 & 6.043e-04 \& 7.075e-02 &  & 6.908e-03 \& 7.727e-01 &  & 3.156e-03 \& 1.670e-01 &  \\
                    32 & 7.528e-05 \& 1.786e-02 & 3.00 \& 1.99 & 1.529e-03 \& 3.768e-01 & 2.18 \& 1.04 & 7.856e-04 \& 8.441e-02 & 2.01 \& 0.98 \\
                    64 & 9.359e-06 \& 4.490e-03 & 3.01 \& 1.99 & 3.659e-04 \& 1.870e-01 & 2.06 \& 1.01 & 1.840e-04 \& 4.239e-02 & 2.09 \& 0.99 \\
                    128 & 1.168e-06 \& 1.126e-03 & 3.00 \& 2.00 & 8.965e-05 \& 9.320e-02 & 2.03 \& 1.00 & 3.422e-05 \& 2.123e-02 & 2.43 \& 1.00 \\ \hline
				\end{tabular}
			}
		\end{center}
	\end{table}
		
	\subsection{Tests for the parameter $K$}  \label{test3}
	In this subsection, we test the accuracy under different settings of hydraulic conductivity $K$. Since we have tested the case $K=1.0$ in the previous tests, we let $K = 10^{-6}$. For other key parameters, we fix $\nu=0.3$ and $c_0=1.0$.

From Table \ref{table_coupled3} to Table \ref{table_id103}, we report numerical results based on the coupled algorithm, the TE decoupled algorithm, the iterative decoupled algorithm with different numbers of iterations, respectively.
By comparing the results in Table \ref{table_coupled3} to Table \ref{table_id103} with those in Table \ref{table_coupled} to Table \ref{table_id10}, it is true that the numerical errors become larger when $K$ is small. However, there is no essential difference in energy-norm error orders for all algorithms. This means that the accuracy of the algorithms is not very sensitive to the hydraulic conductivity $K$. For the iterative decoupled algorithm, by comparing the results in Table \ref{table_id53} with those in Table \ref{table_id103}, we again observe that increasing the number of iterations will lead to better convergence orders. Moreover, the iterative decoupled algorithm gives an optimal order of $L^2-$ norm errors for $\bm{u}$, while other algorithms can not give an optimal $L^2-$ norm errors for $\bm{u}$ under the same parameter setting.
		
	\begin{table}[H]
		\begin{center}
			\caption{Convergence rate of the coupled algorithm for $K=10^{-6}$, $\Delta t = 10^{-3}$.}\label{table_coupled3}
			\centering
			{\scriptsize
				\begin{tabular}{cclclcl}
					\hline
					1/h                     & \multicolumn{1}{l}{$L^2$\& $H^1$ errors of $\ub$} & \multicolumn{1}{l}{Orders} & \multicolumn{1}{l}{$L^2$ \& $H^1$ errors of $\xi$} & \multicolumn{1}{l}{Orders} & \multicolumn{1}{l}{$L^2$\& $H^1$ errors of $p$} & \multicolumn{1}{l}{Orders} \\ \hline
					16 & 1.281e-03 \& 7.126e-02 &  & 6.016e-03 \& 4.880e-01 &  & 7.673e-03 \& 2.113e-01 &  \\
                    32 & 2.966e-04 \& 1.803e-02 & 2.11 \& 1.98 & 1.449e-03 \& 2.382e-01 & 2.05 \& 1.03 & 1.937e-03 \& 9.351e-02 & 1.99 \& 1.18 \\
                    64 & 7.207e-05 \& 4.535e-03 & 2.04 \& 1.99 & 3.560e-04 \& 1.176e-01 & 2.03 \& 1.02 & 4.833e-04 \& 4.453e-02 & 2.00 \& 1.07 \\
                    128 & 1.716e-05 \& 1.137e-03 & 2.07 \& 2.00 & 8.649e-05 \& 5.846e-02 & 2.04 \& 1.01 & 1.167e-04 \& 2.173e-02 & 2.05 \& 1.04 \\  \hline
				\end{tabular}
			}
		\end{center}
	\end{table}

	\begin{table}[H]
		\begin{center}
			\caption{Convergence rate of the TE decoupled algorithm for $K=10^{-6}$, $\Delta t = 10^{-3}$.} \label{table_decoupled3}
			\centering
			{\scriptsize
				\begin{tabular}{cclclcl}
					\hline
					1/h                     & \multicolumn{1}{l}{$L^2$\& $H^1$ errors of $\ub$} & \multicolumn{1}{l}{Orders} & \multicolumn{1}{l}{$L^2$ \& $H^1$ errors of $\xi$} & \multicolumn{1}{l}{Orders} & \multicolumn{1}{l}{$L^2$\& $H^1$ errors of $p$} & \multicolumn{1}{l}{Orders} \\ \hline
					16 & 1.245e-03 \& 7.125e-02 &  & 5.921e-03 \& 4.881e-01 &  & 7.591e-03 \& 2.114e-01 &  \\
                    32 & 2.885e-04 \& 1.802e-02 & 2.11 \& 1.98 & 1.426e-03 \& 2.383e-01 & 2.05 \& 1.03 & 1.879e-03 \& 9.358e-02 & 2.01 \& 1.18 \\
                    64 & 1.496e-04 \& 4.577e-03 & 0.95 \& 1.98 & 5.741e-04 \& 1.177e-01 & 1.31 \& 1.02 & 5.158e-04 \& 4.469e-02 & 1.86 \& 1.07 \\
                    128 & 1.516e-04 \& 1.342e-03 & -0.02 \& 1.77 & 5.273e-04 \& 5.860e-02 & 0.12 \& 1.01 & 3.325e-04 \& 2.232e-02 & 0.63 \& 1.00 \\ \hline
				\end{tabular}
			}
		\end{center}
	\end{table}

	\begin{table}[H]
		\begin{center}
			\caption{Convergence rate of the iterative decoupled algorithm for $K=10^{-6}$, $\Delta t = 5 \times 10^{-3}$, and $iter = 5$.}\label{table_id53}
			\centering
			{\scriptsize
				\begin{tabular}{cclclcl}
					\hline
					1/h                     & \multicolumn{1}{l}{$L^2$\& $H^1$ errors of $\ub$} & \multicolumn{1}{l}{Orders} & \multicolumn{1}{l}{$L^2$ \& $H^1$ errors of $\xi$} & \multicolumn{1}{l}{Orders} & \multicolumn{1}{l}{$L^2$\& $H^1$ errors of $p$} & \multicolumn{1}{l}{Orders} \\ \hline
					16 & 1.277e-03 \& 7.126e-02 &  & 5.999e-03 \& 4.877e-01 &  & 7.640e-03 \& 2.104e-01 &  \\
                    32 & 3.029e-04 \& 1.803e-02 & 2.08 \& 1.98 & 1.475e-03 \& 2.382e-01 & 2.02 \& 1.03 & 1.994e-03 \& 9.368e-02 & 1.94 \& 1.17 \\
                    64 & 8.406e-05 \& 4.541e-03 & 1.85 \& 1.99 & 4.037e-04 \& 1.177e-01 & 1.87 \& 1.02 & 5.777e-04 \& 4.508e-02 & 1.79 \& 1.06 \\
                    128 & 3.594e-05 \& 1.151e-03 & 1.23 \& 1.98 & 1.631e-04 \& 5.862e-02 & 1.31 \& 1.01 & 2.498e-04 \& 2.298e-02 & 1.21 \& 0.97 \\ \hline
				\end{tabular}
			}
		\end{center}
	\end{table}

	\begin{table}[H]
		\begin{center}
			\caption{Convergence rate of the iterative decoupled algorithm for $K=10^{-6}$, $\Delta t = 10^{-2}$, and $iter = 10$.}\label{table_id103}
			\centering
			{\scriptsize
				\begin{tabular}{cclclcl}
					\hline
					1/h                     & \multicolumn{1}{l}{$L^2$\& $H^1$ errors of $\ub$} & \multicolumn{1}{l}{Orders} & \multicolumn{1}{l}{$L^2$ \& $H^1$ errors of $\xi$} & \multicolumn{1}{l}{Orders} & \multicolumn{1}{l}{$L^2$\& $H^1$ errors of $p$} & \multicolumn{1}{l}{Orders} \\ \hline
					16 & 1.273e-03 \& 7.126e-02 &  & 5.993e-03 \& 4.879e-01 &  & 7.617e-03 \& 2.111e-01 &  \\
                    32 & 2.873e-04 \& 1.802e-02 & 2.15 \& 1.98 & 1.425e-03 \& 2.382e-01 & 2.07 \& 1.03 & 1.882e-03 \& 9.334e-02 & 2.02 \& 1.18 \\
                    64 & 6.249e-05 \& 4.531e-03 & 2.20 \& 1.99 & 3.322e-04 \& 1.176e-01 & 2.10 \& 1.02 & 4.294e-04 \& 4.438e-02 & 2.13 \& 1.07 \\
                    128 & 7.558e-06 \& 1.134e-03 & 3.05 \& 2.00 & 6.677e-05 \& 5.845e-02 & 2.31 \& 1.01 & 6.748e-05 \& 2.163e-02 & 2.67 \& 1.04 \\ \hline
				\end{tabular}
			}
		\end{center}
	\end{table}
	
\subsection{Tests for the parameter $c_0$}
In this subsection, we want to check the effects of specific storage coefficient $c_0$ on the accuracy. According to the analysis in Section 4, when $c_0=0$, the convergence rate of the iterative decoupled algorithms may be affected. To check this, we let $c_0 = 0$ and fix $\nu=0.3$ and $K=1$.
	
In Table \ref{table_coupled4} and \ref{table_decoupled4}, we report numerical results based on the coupled algorithm and the TE decoupled algorithm respectively. As we use a relatively large time step size, the error orders of $\bm{u}$ by the TE decoupled algorithm are not optimal. From Table \ref{table_coupled4}, the energy-norm errors based on the coupled algorithm are still of the optimal order. For comparisons, we report the numerical results based on the iterative decoupled algorithm in Table \ref{table_id54} and Table \ref{table_id104}. By comparing Table \ref{table_id54} and Table \ref{table_id104} with Table \ref{table_decoupled4}, we see clearly that the iterative decoupled algorithm gives better results than those of the TE decoupled algorithm. Furthermore, increasing the number of iterations improves the accuracy. When $iter=10$, we see clearly the energy-norm errors are optimal. By comparing Table \ref{table_id54} with Table \ref{table_id5} (and Table \ref{table_id104} with Table \ref{table_id10}), we see that when $c_0=0$, the errors orders for all variables deteriorate a little bit for the iterative decoupled algorithm. However, by increasing the number of iterations, the errors for all variables based on the iterative decoupled algorithm are also of the optimal orders when $c_0=0$. This verifies our analysis (particularly, {\bf Remark \ref{remark2}}) for the iterative decoupled algorithm.
	
		\begin{table}[H]
		\begin{center}
			\caption{Convergence rate of the coupled algorithm. $c_0=0$, $\Delta t = 10^{-3}$.}\label{table_coupled4}
			\centering
			{\scriptsize
				\begin{tabular}{cclclcl}
					\hline
					1/h                     & \multicolumn{1}{l}{$L^2$\& $H^1$ errors of $\ub$} & \multicolumn{1}{l}{Orders} & \multicolumn{1}{l}{$L^2$ \& $H^1$ errors of $\xi$} & \multicolumn{1}{l}{Orders} & \multicolumn{1}{l}{$L^2$\& $H^1$ errors of $p$} & \multicolumn{1}{l}{Orders} \\ \hline
					16 & 1.518e-03 \& 7.149e-02 &  & 7.055e-03 \& 4.744e-01 &  & 1.043e-02 \& 1.752e-01 &  \\
                    32 & 3.574e-04 \& 1.809e-02 & 2.09 \& 1.98 & 1.723e-03 \& 2.348e-01 & 2.03 \& 1.01 & 2.614e-03 \& 8.545e-02 & 2.00 \& 1.04 \\
                    64 & 8.676e-05 \& 4.550e-03 & 2.04 \& 1.99 & 4.231e-04 \& 1.168e-01 & 2.03 \& 1.01 & 6.465e-04 \& 4.251e-02 & 2.02 \& 1.01 \\
                    128 & 2.041e-05 \& 1.140e-03 & 2.09 \& 2.00 & 1.013e-04 \& 5.825e-02 & 2.06 \& 1.00 & 1.535e-04 \& 2.124e-02 & 2.07 \& 1.00 \\ \hline
				\end{tabular}
			}
		\end{center}
	\end{table}

	\begin{table}[H]
		\begin{center}
			\caption{Convergence rate of the TE decoupled algorithm. $c_0=0$, $\Delta t = 10^{-3}$.} \label{table_decoupled4}
			\centering
			{\scriptsize
				\begin{tabular}{cclclcl}
					\hline
					1/h                     & \multicolumn{1}{l}{$L^2$\& $H^1$ errors of $\ub$} & \multicolumn{1}{l}{Orders} & \multicolumn{1}{l}{$L^2$ \& $H^1$ errors of $\xi$} & \multicolumn{1}{l}{Orders} & \multicolumn{1}{l}{$L^2$\& $H^1$ errors of $p$} & \multicolumn{1}{l}{Orders} \\ \hline
					16 & 1.678e-03 \& 7.161e-02 &  & 7.644e-03 \& 4.751e-01 &  & 1.121e-02 \& 1.772e-01 &  \\
                    32 & 4.419e-04 \& 1.814e-02 & 1.93 \& 1.98 & 1.975e-03 \& 2.350e-01 & 1.95 \& 1.02 & 2.854e-03 \& 8.577e-02 & 1.97 \& 1.05 \\
                    64 & 2.383e-04 \& 4.674e-03 & 0.89 \& 1.96 & 8.678e-04 \& 1.169e-01 & 1.19 \& 1.01 & 9.503e-04 \& 4.266e-02 & 1.59 \& 1.01 \\
                    128 & 2.259e-04 \& 1.567e-03 & 0.08 \& 1.58 & 7.711e-04 \& 5.839e-02 & 0.17 \& 1.00 & 7.053e-04 \& 2.145e-02 & 0.43 \& 0.99 \\ \hline
				\end{tabular}
			}
		\end{center}
	\end{table}
	
		\begin{table}[H]
		\begin{center}
			\caption{Convergence rate of the iterative decoupled algorithm. $c_0=0$, $\Delta t = 5 \times 10^{-3}$, and $iter = 5$.}\label{table_id54}
			\centering
			{\scriptsize
				\begin{tabular}{cclclcl}
					\hline
					1/h                     & \multicolumn{1}{l}{$L^2$\& $H^1$ errors of $\ub$} & \multicolumn{1}{l}{Orders} & \multicolumn{1}{l}{$L^2$ \& $H^1$ errors of $\xi$} & \multicolumn{1}{l}{Orders} & \multicolumn{1}{l}{$L^2$\& $H^1$ errors of $p$} & \multicolumn{1}{l}{Orders} \\ \hline
					16 & 1.571e-03 \& 7.152e-02 &  & 7.201e-03 \& 4.748e-01 &  & 1.070e-02 \& 1.764e-01 &  \\
                    32 & 4.372e-04 \& 1.815e-02 & 1.85 \& 1.98 & 2.035e-03 \& 2.350e-01 & 1.82 \& 1.01 & 3.218e-03 \& 8.616e-02 & 1.73 \& 1.03 \\
                    64 & 1.971e-04 \& 4.673e-03 & 1.15 \& 1.96 & 8.618e-04 \& 1.169e-01 & 1.24 \& 1.01 & 1.428e-03 \& 4.307e-02 & 1.17 \& 1.00 \\
                    128 & 1.526e-04 \& 1.435e-03 & 0.37 \& 1.70 & 6.309e-04 \& 5.838e-02 & 0.45 \& 1.00 & 1.042e-03 \& 2.191e-02 & 0.45 \& 0.97 \\ \hline
				\end{tabular}
			}
		\end{center}
	\end{table}

	\begin{table}[H]
		\begin{center}
			\caption{Convergence rate of the iterative decoupled algorithm. $c_0=0$, $\Delta t = 10^{-2}$, and $iter = 10$.}\label{table_id104}
			\centering
			{\scriptsize
				\begin{tabular}{cclclcl}
					\hline
					1/h                     & \multicolumn{1}{l}{$L^2$\& $H^1$ errors of $\ub$} & \multicolumn{1}{l}{Orders} & \multicolumn{1}{l}{$L^2$ \& $H^1$ errors of $\xi$} & \multicolumn{1}{l}{Orders} & \multicolumn{1}{l}{$L^2$\& $H^1$ errors of $p$} & \multicolumn{1}{l}{Orders} \\ \hline
					16 & 1.622e-03 \& 7.154e-02 &  & 7.347e-03 \& 4.755e-01 &  & 1.096e-02 \& 1.777e-01 &  \\
                    32 & 3.771e-04 \& 1.810e-02 & 2.11 \& 1.98 & 1.771e-03 \& 2.350e-01 & 2.05 \& 1.02 & 2.694e-03 \& 8.571e-02 & 2.02 \& 1.05 \\
                    64 & 8.331e-05 \& 4.546e-03 & 2.18 \& 1.99 & 4.092e-04 \& 1.168e-01 & 2.11 \& 1.01 & 6.120e-04 \& 4.252e-02 & 2.14 \& 1.01 \\
                    128 & 1.212e-05 \& 1.135e-03 & 2.78 \& 2.00 & 7.674e-05 \& 5.825e-02 & 2.41 \& 1.00 & 9.558e-05 \& 2.123e-02 & 2.68 \& 1.00 \\ \hline
				\end{tabular}
			}
		\end{center}
	\end{table}

	\section{Conclusions}\label{conclusion}
In this paper, we propose and analyze an iterative decoupled algorithm for Biot model. It is shown that the solution of the iterative decoupled algorithm converges to that of the coupled algorithm. Error analyses are provided for both the coupled algorithm and the iterative decoupled algorithm. Our main conclusion is that the iterative decoupled algorithm is unconditionally stable and convergent. Extensive numerical experiments under different physical parameter settings are provided to verify the performance of the iterative method. By comparing the numerical results obtained by using different algorithms, we conclude that the iterative decoupled algorithm is accurate and efficient.
	
	\appendix
	\section{Error analysis of the coupled algorithm}\label{AppendixA}

The main goal of this appendix is to derive the optimal order error estimate for the coupled algorithm.
In the following lemma, we derive a discrete energy law that mimics the continuous energy law which is proved in \textbf{\bf{Lemma \ref{lemma21}}}.

	
	\begin{lemma}\label{lemma31}
		Let $\{(\bm{u}^{n}_h,\xi^{n}_h,p^{n}_h)\}_{n \geq 0}$ be defined by the coupled algorithm \eqref{c1}-\eqref{c3} with $\theta = 0$, then the following identity holds:
		\begin{align}
			J_h^l + S_h^l = J_h^0, \ \ \ \text{for} \  l\geq 1, \label{clr}
		\end{align}
		where
		\begin{align*}
			J_h^l \coloneqq & \mu \| \varepsilon(\bm{u}^l_h) \|_{L^2(\Omega)}^2 + \frac{1}{2\lambda} \| \alpha p^l_h - \xi_h^l \|_{L^2(\Omega)}^2 + \frac{c_0}{2} \| p^l_h \|_{L^2(\Omega)}^2 - (\bm{f},\bm{u}_h^l ) - \langle \bm{h},\bm{u}_h^l \rangle_{\Gamma_t}, \\
			S_h^l \coloneqq & \Delta t \sum_{n=1}^l \Bigg [
			\Delta t \left( \mu  \| d_t \varepsilon(\bm{u}^n_h) \|_{L^2(\Omega)}^2 + \frac{1}{2\lambda} \| d_t( \alpha p^n_h - \xi_h^n) \|_{L^2(\Omega)}^2  + \frac{c_0}{2} \| d_t p^n_h \|_{L^2(\Omega)}^2 \right) \nonumber \\
			& + K (\nabla p^n_h,\nabla p_h^n) - (Q_s,p_h^n) - \langle g_2,p_h^n \rangle_{\Gamma_f}  \Bigg ].
		\end{align*}
		Here, we denote $d_t \eta^n \coloneqq (\eta^n - \eta^{n-1})/\Delta t$, where $\eta$ can be a vector or a scalar. Moreover,
		\begin{align}
			\|\xi_h^l\|_{L^2(\Omega)} \leq C(\|\varepsilon(\bm{u}_h^l)\|_{L^2(\Omega)}+\|\bm{f}\|_{L^2(\Omega)}+\|\bm{h}\|_{L^2(\Gamma_t)}) \label{bound2},
		\end{align}
		holds with $C$ being a positive constant.
	\end{lemma}
	\begin{proof}[Proof.]
		Setting $\bm{v}_h=d_t \bm{u}_h^n$ in \eqref{c1}, $\phi_h=\xi_h^n$ in \eqref{c2}, and $\psi_h=p_h^n$ in \eqref{c3}, we have
		\begin{align}
			& 2 \mu (\varepsilon(\bm{u}^n_h),d_t  \varepsilon( \bm{u}_h^n))  =  d_t (\bm{f},\bm{u}_h^n) + d_t \langle \bm{h},\bm{u}_h^n \rangle_{\Gamma_t}  +  (\xi^n_h,\mbox{div}d_t \bm{u}_h^n), \label{cl1}  \\
			& (\mbox{div} d_t \bm{u}^n_h ,\xi_h^n) + \frac{1}{\lambda} (d_t \xi^n_h,\xi_h^n) - \frac{\alpha}{\lambda} (d_t p^{n}_h,\xi_h^n) =  0,  \label{cl2}\\
			(c_0 + \frac{\alpha^2}{\lambda}) & \left( d_t p^n_h , p_h^n \right) - \frac{\alpha}{\lambda}\left( d_t \xi^n_h , p_h^n \right)  + K (\nabla p^n_h,\nabla p_h^n)
			=  (Q_s,p_h^n) + \langle g_2,p_h^n \rangle_{\Gamma_f} \label{cl3}.
		\end{align}
		Summing up \eqref{cl1} , \eqref{cl2} and \eqref{cl3}, and then using the identity
		\begin{align}
			2(\eta_h^n, d_t \eta_h^n) = d_t \|\eta_h^n\|_{L^2(\Omega)}^2 + \Delta t \|d_t \eta_h^n\|_{L^2(\Omega)}^2, \label{identity}
		\end{align}
		we have
		\begin{align}
			& d_t\left( \mu \| \varepsilon(\bm{u}^n_h) \|_{L^2(\Omega)}^2 + \frac{1}{2\lambda} \| \alpha p^n_h - \xi_h^n \|_{L^2(\Omega)}^2 + \frac{c_0}{2} \| p^n_h \|_{L^2(\Omega)}^2 - (\bm{f},\bm{u}_h^n ) - \langle \bm{h},\bm{u}_h^n \rangle_{\Gamma_t} \right) \nonumber \\
			+ \ & \Delta t \left( \mu  \| d_t \varepsilon(\bm{u}^n_h) \|_{L^2(\Omega)}^2 + \frac{1}{2\lambda} \| d_t( \alpha p^n_h - \xi_h^n) \|_{L^2(\Omega)}^2  + \frac{c_0}{2} \| d_t p^n_h \|_{L^2(\Omega)}^2 \right)+ K (\nabla p^n_h,\nabla p_h^n) \nonumber \\
			= \ & (Q_s,p_h^n) + \langle g_2,p_h^n\rangle_{\Gamma_f}.
		\end{align}
		Applying the summation operator $\Delta t \sum_{n=1}^l$ to both sides of the above equation, we obtain \eqref{clr}. Furthermore, applying the same techniques used for \eqref{bound}, we derive that \eqref{bound2} holds.
	\end{proof}



Let us introduce some projection operators $\Pi_h^{\bm{V}}:\bm{V} \rightarrow \bm{V}_h$, $\Pi_h^{W}: W \rightarrow W_h$ and $\Pi_h^{M}: M \rightarrow M_h$, satisfying the following equations: for all $(\bm{v}_h,\phi_h,\psi_h) \in \bm{V}_{h} \times W_h \times M_h$,
	\begin{align}
		2 \mu (\varepsilon( \Pi_h^{\bm{V}}\bm{u} ),\varepsilon(\bm{v}_h)) - (\Pi_h^W \xi,\mbox{div}\bm{v}_h) & = 2 \mu (\varepsilon( \bm{u} ),\varepsilon(\bm{v}_h)) - (\xi,\mbox{div}\bm{v}_h),   \label{df1} \\
		(\mbox{div} \Pi_h^{\bm{V}}\bm{u},\phi_h) + \frac{1}{\lambda} (\Pi_h^W \xi,\phi_h) & = (\mbox{div} \bm{u},\phi_h) + \frac{1}{\lambda} (\xi,\phi_h), \label{df2} \\
		K (\nabla \Pi_h^M p,\nabla \psi_h) & = K (\nabla p,\nabla \psi_h). \label{df3}
	\end{align}
	Here, we list the properties of the operators $(\Pi_h^{\bm{V}}, \Pi_h^{W} , \Pi_h^{M})$ \cite{brenner2002mathematical,oyarzua2016locking}. For all $(\bm{u},\xi,p) \in \bm{H}^3(\Omega) \times H^2(\Omega) \times H^2(\Omega)$, there holds
	\begin{align}
		\|\Pi_h^{\bm{V}} \bm{u} - \bm{u}\|_{L^2(\Omega)} + h\|\nabla ( \Pi_h^{\bm{V}}
		\bm{u} - \bm{u} )\|_{L^2(\Omega)} & \leq C h^2  \|\bm{u}\|_{H^3(\Omega)}, \label{bd1} \\
		\|\Pi_h^{W} \xi - \xi \|_{L^2(\Omega)} & \leq C h^2  \|\xi \|_{H^2(\Omega)} , \label{bd2} \\
		\|\Pi_h^{M} p - p \|_{L^2(\Omega)} + h\|\nabla ( \Pi_h^{M} p - p )\|_{L^2(\Omega)} & \leq C h \| p \|_{H^2(\Omega)}. \label{bd3}
  	\end{align}
	For convenience, we introduce the following notations:
	\begin{align*}
		& e_{\bm{u}}^n = \bm{u}^n - \bm{u}_h^n = (\bm{u}^n - \Pi_h^{\bm{V}}\bm{u}^n) +( \Pi_h^{\bm{V}}\bm{u}^n - \bm{u}_h^n) \coloneqq e_{\bm{u}}^{I,n} + e_{\bm{u}}^{h,n}, \\
		& e_{\xi}^n    = \xi^n    - \xi_h^n    = (\xi^n - \Pi_h^W\xi^n) + (\Pi_h^W\xi^n  - \xi_h^n)\coloneqq e_{\xi}^{I,n} + e_{\xi}^{h,n}, \\
		& e_{p}^n      = p^n      - p_h^n      = (p^n - \Pi_h^M p^n) + (\Pi_h^M p^n   - p_h^n) \coloneqq e_{p}^{I,n} + e_{p}^{h,n}.
	\end{align*}
	
 	For the error estimates, we need to evaluate some error terms.
 	
 	\begin{lemma}
 	    Let $\{(\bm{u}^{n}_h,\xi^{n}_h,p^{n}_h)\}_{n \geq 0}$ be defined by the coupled algorithm \eqref{c1}-\eqref{c3} with $\theta = 0$,  then we have the following identity:
		\begin{align}
			E_h^l + \Delta t & \sum_{n=1}^l K  \| \nabla e_p^{h,n} \|_{L^2(\Omega)}
			\nonumber \\
			+ & (\Delta t)^2 \sum_{n=1}^l  \bigg( \mu \| d_t \varepsilon( e_{\bm{u}}^{h,n} ) \|_{L^2(\Omega)}^2 + \frac{1}{2\lambda}\|d_t (\alpha e_p^{h,n} - e_{\xi}^{h,n})\|_{L^2(\Omega)}^2 + \frac{c_0}{2}\|d_t e_p^{h,n}\|_{L^2(\Omega)}^2	\bigg)
			\nonumber \\
			= E_h^0 + & \Delta t \sum_{n=1}^l \Big[ ( \mbox{div} ( d_t \bm{u}^{n} - \bm{u}_t^n ),e_{\xi}^{h,n} ) + \frac{1}{\lambda} (d_t \xi^{n} - \xi_t^n, e_{\xi}^{h,n}) - \frac{\alpha}{\lambda} (d_t p^{n}-p_t^n,e_{\xi}^{h,n})
			\nonumber \\
			& + (c_0 + \frac{\alpha^2}{\lambda})\left( d_t  \Pi_M^h p^{n} - p_t^n , e_p^{h,n} \right) - \frac{\alpha}{\lambda}(d_t \Pi_W^h \xi^{n} - \xi_t^n , e_p^{h,n})
			\Big], \label{lema2eq}
		\end{align}
		where
		\begin{align}
			E_h^l \coloneqq & \mu \| \varepsilon( e_{\bm{u}}^{h,l} ) \|_{L^2(\Omega)}^2 + \frac{1}{2\lambda}\|\alpha e_p^{h,l} - e_{\xi}^{h,l}\|_{L^2(\Omega)}^2 + \frac{c_0}{2}\|e_p^{h,l}\|_{L^2(\Omega)}^2.
		\end{align}
    \end{lemma}
    \begin{proof}
        First, we use \eqref{c1}, \eqref{vp1} and \eqref{df1} to get
        \begin{align}
			2 \mu (\varepsilon( e_{\bm{u}}^{h,n} ),\varepsilon(\bm{v}_h)) - (e_{\xi}^{h,n},\mbox{div}\bm{v}_h) = 0. \label{lema21}
		\end{align}
		The combination of \eqref{c2} with $\theta = 0$, \eqref{vp2}, and \eqref{df2} implies that
		\begin{align}
		    & \left( \mbox{div} ( d_t e_{\bm{u}}^{h,n} ),\phi_h \right) + \frac{1}{\lambda} (d_t e_{\xi}^{h,n}, \phi_h) - \frac{\alpha}{\lambda} (d_t e_{p}^{h,n} ,\phi_h) \nonumber
			\\
			& \ \ \ = \left( \mbox{div} ( d_t \bm{u}^{n} - \bm{u}_t^n ),\phi_h \right) + \frac{1}{\lambda} (d_t \xi^{n} - \xi_t^n, \phi_h) - \frac{\alpha}{\lambda} (d_t \Pi_M^h p^{n}-p_t^n,\phi_h).\label{lema22}
		\end{align}
		Using \eqref{c3}, \eqref{vp3} and \eqref{df3}, we obtain
		\begin{align}
			& (c_0 + \frac{\alpha^2}{\lambda})\left( d_t  e_p^{h,n} , \psi_h \right) - \frac{\alpha}{\lambda}(d_t e_{\xi}^{h,n} , \psi_h) + K (\nabla e_p^{h,n},\nabla \psi_h) \nonumber
			\\
			& \ \ \ = (c_0 + \frac{\alpha^2}{\lambda})\left( d_t  \Pi_M^h p^{n} - p_t^n , \psi_h \right) - \frac{\alpha}{\lambda}(d_t \Pi_W^h \xi^{n} - \xi_t^n , \psi_h).\label{lema23}
		\end{align}
		Setting $\bm{v}_h = d_t e_{\bm{u}}^{h,n}$ in \eqref{lema21}, $\phi_h = e_{\xi}^{h,n}$ in \eqref{lema22} and $\psi_h = e_p^{h,n}$ in \eqref{lema23} and adding the resulted equations together, we derive
		\begin{align}
			2 \mu (\varepsilon( e_{\bm{u}}^{h,n} ),\varepsilon(d_t & e_{\bm{u}}^{h,n}))
			+ \frac{1}{\lambda} (d_t e_{\xi}^{h,n}, e_{\xi}^{h,n}) - \frac{\alpha}{\lambda} (d_t e_{p}^{h,n} ,e_{\xi}^{h,n}) \nonumber
			\\
			& + (c_0 + \frac{\alpha^2}{\lambda})\left( d_t  e_p^{h,n} , e_p^{h,n} \right) - \frac{\alpha}{\lambda}(d_t e_{\xi}^{h,n} , e_p^{h,n}) + K (\nabla e_p^{h,n},\nabla e_p^{h,n}) \nonumber
			\\
			= ( \mbox{div} ( d_t & \bm{u}^{n} - \bm{u}_t^n ),e_{\xi}^{h,n} ) + \frac{1}{\lambda} (d_t \xi^{n} - \xi_t^n, e_{\xi}^{h,n}) - \frac{\alpha}{\lambda} (d_t \Pi_M^h p^{n}-p_t^n,e_{\xi}^{h,n}) \nonumber
			\\
			& \ \ \ + (c_0 + \frac{\alpha^2}{\lambda})\left( d_t  \Pi_M^h p^{n} - p_t^n , e_p^{h,n} \right) - \frac{\alpha}{\lambda}(d_t \Pi_W^h \xi^{n} - \xi_t^n , e_p^{h,n}).
		\end{align}
		Using the identity \eqref{identity}, we derive that
		\begin{align}
		    \ d_t  \bigg(
			\mu \| \varepsilon( e_{\bm{u}}^{h,n}& ) \|_{L^2(\Omega)}^2 + \frac{1}{2\lambda}\|\alpha e_p^{h,n} - e_{\xi}^{h,n}\|_{L^2(\Omega)}^2 + \frac{c_0}{2}\|e_p^{h,n}\|_{L^2(\Omega)}^2
			\bigg) + K (\nabla e_p^{h,n},\nabla e_p^{h,n}) \nonumber \\
			& + \Delta t \left(
			\mu \| d_t \varepsilon( e_{\bm{u}}^{h,n} ) \|_{L^2(\Omega)}^2 + \frac{1}{2\lambda}\|d_t (\alpha e_p^{h,n} - e_{\xi}^{h,n})\|_{L^2(\Omega)}^2 + \frac{c_0}{2}\|d_t e_p^{h,n}\|_{L^2(\Omega)}^2
			\right)
			\nonumber \\
			= ( & \mbox{div}( d_t \bm{u}^{n} - \bm{u}_t^n ),e_{\xi}^{h,n} ) + \frac{1}{\lambda} (d_t \xi^{n} - \xi_t^n, e_{\xi}^{h,n}) - \frac{\alpha}{\lambda} (d_t \Pi_M^h p^{n}-p_t^n,e_{\xi}^{h,n})
			\nonumber \\
			& \ \ \ \ + (c_0 + \frac{\alpha^2}{\lambda})\left( d_t  \Pi_M^h p^{n} - p_t^n , e_p^{h,n} \right) - \frac{\alpha}{\lambda}(d_t \Pi_W^h \xi^{n} - \xi_t^n , e_p^{h,n}).
		\end{align}
		Applying the summation operator $\Delta t \sum_{n=1}^l$ to both sides, we obtain \eqref{lema2eq}. The proof is complete.
    \end{proof}

The following theorems give the error estimates of the coupled algorithm. For simplicity, $X \lesssim Y$ is used to denote an inequality $X \leq CY$, where $C$ is a positive constant independent of mesh sizes $h$.

	\begin{thm}\label{thm31}
		Let $\{(\bm{u}^{n}_h,\xi^{n}_h,p^{n}_h)\}_{n \geq 0}$ be defined by the coupled algorithm \eqref{c1}-\eqref{c3} with $\theta = 0$, then the following error estimate holds:
		\begin{align}
			\max_{0\leq n\leq l}  \left[ \mu \| \varepsilon( e_{\bm{u}}^{h,n} ) \|_{L^2(\Omega)}^2 + \frac{1}{2\lambda}\|\alpha e_p^{h,n} - e_{\xi}^{h,n} \|_{L^2(\Omega)}^2 + \frac{c_0}{2}\|e_p^{h,n}\|_{L^2(\Omega)}^2 \right] \nonumber \\
			+ \Delta t \sum_{n=0}^l K \|\nabla e_p^{h,n}\|_{L^2(\Omega)}^2
			\leq C_1 (\Delta t)^2 + C_2 h^2,
		\end{align}
		where
		\begin{align}
			& C_1 = C_1(\|\bm{u}_{tt}\|_{L^2(0,t_l;H^1(\Omega))}^2,
			\|\xi_{tt}\|_{L^2(0,t_l;L^2(\Omega))}^2,
			\|p_{tt}\|_{L^2(0,t_l;L^2(\Omega))}^2 ),\\
			& C_2 = C_2(h^2 \|\xi_t\|_{L^2(0,t_l;H^2(\Omega))}^2, \|p_t\|_{L^2(0,t_l;H^2(\Omega))}^2).
		\end{align}
	\end{thm}
	\begin{proof}
	    Discarding the positive terms of the left-hand side in \eqref{lema2eq} and setting $\bm{u}_h^0 = \Pi_h^{\bm{V}} \bm{u}^0$, $\xi_h^0 = \Pi_h^W \xi^0$, and $p_h^0 = \Pi_h^M p^0$, we derive the following inequality
	    \begin{align}
			E_h^l +  \Delta t &  \sum_{n=1}^l K (\nabla e_p^{h,n},\nabla e_p^{h,n})
			\nonumber \\
			\leq \Delta t & \sum_{n=1}^l \bigg[ ( \mbox{div} ( d_t \bm{u}^{n} - \bm{u}_t^n ),e_{\xi}^{h,n} ) + \frac{1}{\lambda} (d_t \xi^{n} - \xi_t^n, e_{\xi}^{h,n}) - \frac{\alpha}{\lambda} (d_t \Pi_M^h p^{n}-p_t^n,e_{\xi}^{h,n})
			\nonumber \\
			& + (c_0 + \frac{\alpha^2}{\lambda})\left( d_t  \Pi_M^h p^{n} - p_t^n , e_p^{h,n} \right) - \frac{\alpha}{\lambda}(d_t \Pi_W^h \xi^{n} - \xi_t^n , e_p^{h,n}) \bigg].
		\end{align}
		Using Taylor series expansion and the Cauchy-Schwarz inequality, we can bound the first term by
		\begin{align}
		    (\Delta t) \sum_{n=1}^l ( \mbox{div} ( d_t \bm{u}^{n} - \bm{u}_t^n ),e_{\xi}^{h,n} )
		    & = \sum_{n=1}^l ( \mbox{div} ( \bm{u}^{n} - \bm{u}^{n-1} - (\Delta t)\bm{u}_t^n ),e_{\xi}^{h,n} )
		    \nonumber \\
		    & \lesssim \sum_{n=1}^l \|\bm{u}^{n} - \bm{u}^{n-1} - (\Delta t)\bm{u}_t^n \|_{H^1(\Omega)} \|e_{\xi}^{h,n}\|_{L^2(\Omega)}
		    \nonumber \\
		    & \lesssim (\Delta t)^2\|\bm{u}_{tt} \|_{L^2(0,t_l;H^1(\Omega))}^2 + (\Delta t) \sum_{n=1}^l \mu \|e_{\xi}^{h,n}\|_{L^2(\Omega)}^2.
		\end{align}
		Similarly, the second term can be bounded by
		\begin{align}
		    \frac{(\Delta t)}{\lambda} \sum_{n=1}^l  (d_t \xi^{n} - \xi_t^n, e_{\xi}^{h,n}) \lesssim (\Delta t)^2\| \xi_{tt} \|_{L^2(0,t_l;L^2(\Omega))}^2 + (\Delta t) \sum_{n=1}^l \mu \|e_{\xi}^{h,n}\|_{L^2(\Omega)}^2.
		\end{align}
		By use of estimate \eqref{bd3}, we see that the third term satisfies
		\begin{align}
		    \frac{\alpha (\Delta t)}{\lambda} \sum_{n=1}^l  (d_t \Pi_M^h p^{n}-p_t^n,e_{\xi}^{h,n})
		    & \lesssim \sum_{n=1}^l \left( \|\Pi_M^h p^{n} -\Pi_M^h p^{n-1} - (\Delta t)p_t^n \|_{L^2(\Omega)} \|e_{\xi}^{h,n}\|_{L^2(\Omega)} \right)
		    \nonumber \\
		    \lesssim \sum_{n=1}^l \Big( \|\Pi_M^h (p^{n} - p^{n-1}) & - (p^{n} - p^{n-1}) \|_{L^2(\Omega)} +  \|p^{n} - p^{n-1}-(\Delta t)p_t^n \|_{L^2(\Omega)}  \Big)\|e_{\xi}^{h,n}\|_{L^2(\Omega)}
		    \nonumber \\
		    \lesssim \big( h^2\|p_t\|_{L^2(0,t_l;H^2(\Omega))}^2 + & (\Delta t)^2\|p_{tt}\|_{L^2(0,t_l;L^2(\Omega))}^2 \big) + (\Delta t) \sum_{n=1}^l \mu  \|e_{\xi}^{h,n}\|_{L^2(\Omega)}^2.
		\end{align}
		Likewise, applying the Poincar\'{e} inequality, we can bound the fourth term and the fifth term by
		\begin{align}
		    (c_0 + \frac{\alpha^2}{\lambda}) ( d_t  \Pi_M^h p^{n} & - p_t^n , e_p^{h,n} ) \lesssim (\Delta t) \sum_{n=1}^l K \|\nabla e_{p}^{h,n}\|_{L^2(\Omega)}^2
		    \nonumber \\
		    + & \ \big( h^2\|p_t\|_{L^2(0,t_l;H^2(\Omega))}^2 + (\Delta t)^2\|p_{tt}\|_{L^2(0,t_l;L^2(\Omega))}^2 \big),
		\end{align}
		\begin{align}
		    \frac{\alpha}{\lambda}(d_t \Pi_W^h \xi^{n} & - \xi_t^n , e_p^{h,n}) \lesssim (\Delta t) \sum_{n=1}^l K \|\nabla e_{p}^{h,n}\|_{L^2(\Omega)}^2
		    \nonumber \\
		    + & \ \big( h^4 \|\xi_t\|_{L^2(0,t_l;H^2(\Omega))}^2  + (\Delta t)^2\|\xi_{tt}\|_{L^2(0,t_l;L^2(\Omega))}^2 \big).
		\end{align}
		The above bounds and the discrete Gronwall's inequality imply that
		\begin{align}
			E_h^l + \Delta t \sum_{n=1}^l K \| \nabla e_p^{h,n} \|_{L^2(\Omega)}^2 \lesssim  (\Delta t)^2 \big( \|\bm{u}_{tt}\|_{L^2(0,t_l;H^1(\Omega))}^2 +
			\|\xi_{tt}\|_{L^2(0,t_l;L^2(\Omega))}^2 +
			\|p_{tt}\|_{L^2(0,t_l;L^2(\Omega))}^2 \big)
			\nonumber \\
			+ \ h^2 \big(  h^2 \|\xi_t\|_{L^2(0,t_l;H^2(\Omega))}^2 + \|p_t\|_{L^2(0,t_l;H^2(\Omega))}^2 \big) .
		\end{align}
		The proof is complete.
	\end{proof}
    \begin{thm}
		Let $\{(\bm{u}^{n}_h,\xi^{n}_h,p^{n}_h)\}_{n \geq 0}$ be defined by the coupled algorithm \eqref{c1}-\eqref{c3} with $\theta = 0$, then the following error estimate holds:
		\begin{align} \label{thm4app}
			\max_{0\leq n\leq l}  \big[ \mu \| \varepsilon( e_{\bm{u}}^{n} ) \|_{L^2(\Omega)}^2 & + \frac{1}{2\lambda}\|\alpha e_p^{n} \|_{L^2(\Omega)}^2 + \frac{c_0}{2}\|e_p^{n}\|_{L^2(\Omega)}^2 \big]
			\nonumber \\
			& + \Delta t \sum_{n=0}^l K \|\nabla e_p^{n}\|_{L^2(\Omega)}^2
			\leq C_1 (\Delta t)^2 + C_2 h^2.
		\end{align}
		Moreover, we have the estimate
		\begin{align}
		    \|e_{\xi}^{n}\|_{L^2(\Omega)} \lesssim \| \varepsilon( e_{\bm{u}}^{n})\|_{L^2(\Omega)}.
		\end{align}
	\end{thm}
	\begin{proof} By use of the discrete inf-sup condition and a Cauchy-Schwarz inequality, we obtain
        \begin{align*}
        \|\alpha e_p^{h,n}\|_{L^2(\Omega)} \leq \|\alpha e_p^{h,n} - e_{\xi}^{h,n}\|_{L^2(\Omega)} + \|e_{\xi}^{h,n}\|_{L^2(\Omega)} \lesssim ( \|\alpha e_p^{h,n} - e_{\xi}^{h,n}\|_{L^2(\Omega)} + \| \varepsilon( e_{\bm{u}}^{n} )\|_{L^2(\Omega)} ).
        \end{align*}
	    Then, the error estimate \eqref{thm4app} follows from a straightforward application of triangle inequalities to
		\begin{align*}
		    e_{\bm{u}}^n = e_{\bm{u}}^{I,n} + e_{\bm{u}}^{h,n} \ \ \mbox{and} \ \
		    e_{p}^n      = e_{p}^{I,n} + e_{p}^{h,n},
	    \end{align*}
		the properties \eqref{bd1}, \eqref{bd3}, and \bf{Theorem \ref{thm31}}.
	\end{proof}
	\bibliographystyle{amsplain}
		\bibliography{references}

\end{document}